\documentclass[12pt]{amsart}
\usepackage[margin=1.0in]{geometry}
\usepackage{amsfonts,amscd,amssymb,enumerate,enumitem,hyperref,cleveref}
\allowdisplaybreaks

\date{\today}

\newtheorem{thm}{Theorem}[section]
\newtheorem{cor}[thm]{Corollary}
\newtheorem{lem}[thm]{Lemma}

\newtheorem{prop}[thm]{Proposition}
\theoremstyle{definition}

\theoremstyle{remark}
\newtheorem{rem}[thm]{Remark}

\numberwithin{equation}{section}

\newcommand{\R}{\mathbb R}

\newcommand{\He}{\mathbb H}

\newcommand{\C}{{\mathbb C}}

\renewcommand{\Im}{\operatorname{Im}}
\newcommand{\tr}{\operatorname{tr}}

\usepackage{color}

\title[Twisted Fock spaces]
{Boundedness of certain  linear operators  \\
on twisted Fock spaces}

\author[R. Garg and S. Thangavelu]{Rahul Garg and Sundaram Thangavelu}

\address[R. Garg]{Department of Mathematics, Indian Institute of Science Education and Research Bhopal, Bhopal--462066, India.}
\email{rahulgarg@iiserb.ac.in}

\address[S. Thangavelu]{Department of Mathematics, Indian Institute of Science, Bangalore--560012, India.}
\email{veluma@iisc.ac.in}


\begin{document}

\maketitle

\begin{abstract} On the twisted Fock spaces $ \mathcal{F}^\lambda(\C^{2n}) $ we consider two types of convolution operators $ S_\varphi $ and $ \widetilde{S}_\varphi $ associated to an element $ \varphi \in  \mathcal{F}^\lambda(\C^{2n}) .$ We find a necessary and sufficient condition on $ \varphi $ so that $ S_\varphi $ (resp. $ \widetilde{S}_\varphi $ ) is bounded on $ \mathcal{F}^\lambda(\C^{2n}).$ We show that for any given non constant $ \varphi $ at least one of these two operators is unbounded.
\end{abstract}
 
%%%%%%%%%%%%%%%%%%%%%%%%%%%%%%%%%%%%%%
%%%%%%%%%%%%%%%%%%%%%%%%%%%%%%%%%%%%%%
%%%%%%%%%%%%%%%%%%%%%%%%%%%%%%%%%%%%%%
%%%%%%%%%%%%%%%%%%%%%%%%%%%%%%%%%%%%%%
%%%%%%%%%%%%%%%%%%%%%%%%%%%%%%%%%%%%%%
%%%%%%%%%%%%%%%%%%%%%%%%%%%%%%%%%%%%%%
%%%%%%%%%%%%%%%%%%%%%%%%%%%%%%%%%%%%%%
%%%%%%%%%%%%%%%%%%%%%%%%%%%%%%%%%%%%%%
%%%%%%%%%%%%%%%%%%%%%%%%%%%%%%%%%%%%%%
%%%%%%%%%%%%%%%%%%%%%%%%%%%%%%%%%%%%%%

\section{Introduction} \label{Sec-intro}
In this article we consider a family of Fock spaces $ \mathcal{F}^\lambda(\C^{2n}) $ indexed by $ \lambda \in \R $ which coincide with the classical Fock space $\mathcal{F}(\C^{2n})$ when $ \lambda =0.$ These spaces are related to the twisted Bergman spaces $ \mathcal{B}_t^\lambda(\C^{2n}) $ studied in \cite{KTX} which appear naturally in connection with the heat kernel transform or the Segal-Bargmann transform on the Heisenberg group $ \He^n.$ Let us consider the weight function
$$ w_\lambda(z,w) = e^{\lambda \Im( z\cdot \bar{w})} \, e^{-\frac{1}{2} \lambda (\coth \lambda) |(z,w)|^2},\,\,\, (z,w) \in \C^{2n}.$$
Then $ \mathcal{F}^\lambda(\C^{2n}) $ is defined as the space of all entire functions $ F $ on $\C^{2n} $ which are square integrable with respect to the measure $ w_\lambda(z,w) \, dz \, dw.$ We equip $ \mathcal{F}^\lambda(\C^{2n}) $ with the norm
$$ \|F \|_{\mathcal{F}^\lambda}^2 = \int_{\C^{2n}} |F(z,w)|^2 \, w_\lambda(z,w) \, dz\, dw. $$ 
Note that when $ \lambda =0, \, w_0(z,w) = e^{-\frac{1}{2}  |(z,w)|^2}$ and hence  $ \mathcal{F}^0(\C^{2n}) = \mathcal{F}(\C^{2n}),$ the standard Fock space of entire functions on $ \C^{2n} $ which are square integrable with respect to the Gaussian measure $ d\nu(z,w) = w_0(z,w) \, dz \, dw = e^{-\frac{1}{2} |(z,w)|^2} \, dz \, dw.$\\

In \cite{CLSWY, CHLS, Thangavelu-arxiv-Fock-Sobolev-2023} the authors have studied  boundedness properties of operators $S_\varphi$ of the form
$$ S_\varphi F(z) = \int_{\C^n} F(w)\, \varphi(z-\bar{w})\, e^{\frac{1}{2} z \cdot \bar{w}} \, e^{-\frac{1}{2}|w|^2} \, dw $$ 
on the Fock spaces $ \mathcal{F}(\C^n)$, for $\varphi \in \mathcal{F}(\C^n)$. They have proved that $ S_\varphi $ is bounded on $ \mathcal{F}(\C^n) $ if and only if $\varphi = G m$ for some $ m \in L^\infty(\R^n).$ Here $ G: L^\infty(\R^n) \rightarrow \mathcal{F}(\C^n) $ is the Gauss-Bargmann transform defined by
$$ Gm(z) = e^{\frac{1}{4}z^2} \, \int_{\R^n} m(\xi) \,e^{i z\cdot \xi}\, e^{-|\xi|^2} \, d\xi.$$
One of the aims of this article is to study analogues of $ S_\varphi $ on the twisted Fock spaces $\mathcal{F}^\lambda(\C^{2n})$, for $\lambda \in \R^\ast= \mathbb{R} \setminus \left\lbrace 0\right\rbrace$. For each $ \varphi \in  \mathcal{F}^\lambda(\C^{2n}) $ we define
\begin{align} \label{def:convolution-operator-tiwsted-fock}
S_\varphi F(z,w) = \int_{\C^{2n}} F(a,b)  \varphi(z-\bar{a}, w-\bar{b}) e^{\frac{1}{2}\lambda (\coth \lambda)(z \cdot \bar{a}+w \cdot \bar{b})} e^{-\frac{i}{2} \lambda (w \cdot \bar{a}- z \cdot \bar{b})} w_\lambda(a,b) \, da \, db. 
\end{align}
In the setting of twisted Fock spaces, the analogue of the Gauss-Bargmann transform $ G $ is given by the map $ G_\lambda: B(L^2(\R^n)) \rightarrow  \mathcal{F}^\lambda(\C^{2n}) $ defined by
\begin{align} \label{def:Gauss-Bargmann-transform-tiwsted-fock}
G_\lambda(M)(z,w) = p_1^\lambda(z,w)^{-1}\, \tr\left( \pi_\lambda(-z,-w) e^{-\frac{1}{2} H(\lambda)} M e^{-\frac{1}{2}H(\lambda)} \right) .
\end{align}
In this article, we consider the boundedness of the operators $ S_\varphi $ on twisted Fock spaces $\mathcal{F}^\lambda(\C^{2n})$ and establish the following result. 

\begin{thm} \label{thm:convolution-operator-twisted-fock}
For any $ \varphi \in \mathcal{F}^\lambda(\C^{2n}) $  the operator $ S_\varphi $ defined in \eqref{def:convolution-operator-tiwsted-fock} is bounded on $\mathcal{F}^\lambda(\C^{2n}) $ if and only if $ \varphi = G_\lambda(M)$ for some $ M \in B(L^2(\R^n)).$
\end{thm} 

As a corollary to Theorem \ref{thm:convolution-operator-twisted-fock} we can obtain the following representation formula for the operators $ S_\varphi.$ Let  $ \Gamma_{1/2} = e^{-\frac{1}{2}H(\lambda)} $ be the Hermite semigroup and define $ \mathcal{S}(\Gamma_{1/2}) $ to be  the space of all operators $ T $ on $ L^2(\R^n) $ such that $ e^{-\frac{1}{2}H(\lambda)} T$ is Hilbert-Schmidt. We equip this space with the norm $ \|T\| = \left\| e^{-\frac{1}{2}H(\lambda)} T \right\|_{HS}.$ Then  $ G_\lambda $ has a natural extension to  $ \mathcal{S}(\Gamma_{1/2}) $ as an isometry onto  $ \mathcal{F}^\lambda(\C^{2n}).$ Let $ G_\lambda^\ast $ denote the adjoint of $ G_\lambda.$

\begin{cor} \label{cor:convolution-operator-twisted-fock-via-Gauss-Bargmann-transform} 
For any bounded linear operator $ S_\varphi $ on the twisted Fock space $\mathcal{F}^\lambda(\C^{2n})$ let $ M $ be the associated operator as in Theorem \ref{thm:convolution-operator-twisted-fock}. Then
$ S_\varphi F = G_\lambda \left( G_\lambda^\ast F  \circ M\right).$
\end{cor}

Corollary \ref{cor:convolution-operator-twisted-fock-via-Gauss-Bargmann-transform} suggests that when $ \varphi $ and $ \psi $ belong to the image of $ B(L^2(\R^n)) $ under $ G_\lambda $, we use the notation $ \varphi \ast_\lambda \psi = S_\psi \varphi.$ Then we can rewrite the result of the corollary as $ G_\lambda^\ast(\varphi \ast_\lambda \psi) = G_\lambda^\ast(\varphi) \, G_\lambda^\ast(\psi).$ Under this convolution the subspace $ \mathcal{A}^\lambda(\C^{2n}) = G_\lambda(B(L^2(\R^n))) \subset \mathcal{F}^\lambda(\C^{2n}) $ becomes a Banach algebra which is non-commutative. We shall come back to this discussion in Section \ref{subsection-algebra-entire-functions}. \\

Along with the operators $ S_\varphi$ we also consider the operators $ \widetilde{S}_\varphi $ on  $\mathcal{F}^\lambda(\C^{2n})$  defined by
\begin{align} \label{def:tilde-convolution-operator-tiwsted-fock}
\widetilde{S}_\varphi F(z,w) = \int_{\C^{2n}} F(a,b)  \varphi(z+\bar{a}, w+\bar{b}) e^{\frac{1}{2}\lambda (\coth \lambda)(z \cdot \bar{a}+w \cdot \bar{b})} e^{-\frac{i}{2} \lambda (w \cdot \bar{a}- z \cdot \bar{b})} w_\lambda(a,b) \, da \,db . 
\end{align} 
As in the classical setting, these two classes of operators are related via the unitary operator $ U $ defined by $ UF(z,w) = F(-iz,-iw) $ on $ \mathcal{F}^\lambda(\C^{2n}).$ Indeed, it is easy to see from \eqref{def:Gauss-Bargmann-transform-tiwsted-fock} and \eqref{def:tilde-convolution-operator-tiwsted-fock} that 
\begin{align} \label{relation:two-convolutions-operators} 
U^\ast \circ S_\varphi \circ U = \widetilde{S}_{U^\ast \varphi}. 
\end{align} 
Consider the following family of operators $ \rho_\lambda(a,b)$ acting on $ \mathcal{F}^\lambda(\C^{2n}) $ as follows:
\begin{align} \label{def:operators-rho-lambda}
\rho_\lambda(a,b)F(z,w) = e^{-i\frac{\lambda}{2}( w\cdot \bar{a}- z \cdot \bar{b})}\, e^{\frac{1}{2}\lambda (\coth \lambda)( z\cdot \bar{a}+w \cdot \bar{b})} \, F(z-a,w-b), 
\end{align} 
for $(a,b) \in \C^{2n}$. It turns out that the operators $ S_\varphi$ commute with $ \rho_\lambda(a,b)$ for all $ (a, b)\in \R^{2n}.$ Likewise, in view of the relation \eqref{relation:two-convolutions-operators}, the operators $ \widetilde{S}_\varphi $ commute with $ U^\ast \circ \rho_\lambda(a,b) \circ U$ for all $ (a,b) \in \R^{2n}.$ An easy calculation shows that for $ (a, b) \in \R^{2n},  U^\ast \circ \rho_\lambda(a,b) \circ U = \rho_\lambda(-ia,-ib) $ and hence $ \widetilde{S}_\varphi $ commutes with $\rho_\lambda(ia,ib)$ for all $ (a, b) \in \R^{2n}.$ More importantly, it can be shown that if a bounded linear operator $S$ on the twisted Fock space $ \mathcal{F}^\lambda(\C^{2n}) $  commutes with $ \rho_\lambda(a,b)$ for all $ (a, b)\in \R^{2n},$ then $S = S_\varphi$ for some $\varphi \in \mathcal{F}^\lambda(\C^{2n})$. Analogous fact holds true for bounded linear operators on $ \mathcal{F}^\lambda(\C^{2n}) $ which commute with $ \rho_\lambda(ia, ib)$ for all $ (a, b)\in \R^{2n} $. We shall come back to this fact later, see \Cref{prop:converse-statement-bdd-linear-op-twisted-fock-space-commuting-with-twisted-translation}. \\

As in the classical setting (see \cite{Thangavelu-arxiv-Fock-Sobolev-2023}), there is an uncertainty regarding the simultaneous boundedness of $ S_\varphi $ and $ \widetilde{S}_\varphi.$ 

\begin{thm} \label{thm:uncertainty-twisted-fock-spaces}
For any non constant $ \varphi \in \mathcal{F}^\lambda (\C^{2n})$ at least one of the operators $ S_\varphi $ and $ \widetilde{S}_\varphi $ fails to be bounded on $ \mathcal{F}^\lambda(\C^{2n}).$
\end{thm} 

The analogue of this theorem in the classical setting is proved by appealing to Hardy's theorem for the Fourier transform on $ \R^n,$ see \cite{Thangavelu-arxiv-Fock-Sobolev-2023} for the proof.  In a similar vein, the above theorem is proved by means of Hardy's theorem for the Weyl transform. 
As we have $ U^\ast \circ S_{U \varphi} \circ U = \widetilde{S}_{ \varphi},$ the above theorem can be viewed as the following statement  about the subspace $ \mathcal{A}^\lambda(\C^{2n}) = G_\lambda(B(L^2(\R^n))) \subset \mathcal{F}^\lambda(\C^{2n}) $ under the action of $ U$: if $ \varphi \in \mathcal{A}^\lambda(\C^{2n}),$ then $ U\varphi $ never belongs to $ \mathcal{A}^\lambda(\C^{2n}) $ unless $\varphi$  is a constant. \\ 

The paper is organised as follows. In Section \ref{Sec-prelim} we recall preliminaries which are relevant to this work. In Section \ref{Sec-Weyl-multipliers-Fock-spaces} we establish Theorem \ref{thm:convolution-operator-twisted-fock} which talks about the boundedness of operators $ S_\varphi $ on twisted Fock spaces $\mathcal{F}^\lambda(\C^{2n})$. Finally, making use of an analogue of Hardy's theorem for the Weyl transform \cite[Theorem 2.9.5]{Book-Thangavelu-uncertainty}, we develop the proof of Theorem \ref{thm:uncertainty-twisted-fock-spaces} in Section \ref{Sec-proof-main-results}.

%%%%%%%%%%%%%%%%%%%%%%%%%%%%%%%%%%%%%%
%%%%%%%%%%%%%%%%%%%%%%%%%%%%%%%%%%%%%%
%%%%%%%%%%%%%%%%%%%%%%%%%%%%%%%%%%%%%%
%%%%%%%%%%%%%%%%%%%%%%%%%%%%%%%%%%%%%%
%%%%%%%%%%%%%%%%%%%%%%%%%%%%%%%%%%%%%%
%%%%%%%%%%%%%%%%%%%%%%%%%%%%%%%%%%%%%%
%%%%%%%%%%%%%%%%%%%%%%%%%%%%%%%%%%%%%%
%%%%%%%%%%%%%%%%%%%%%%%%%%%%%%%%%%%%%%
%%%%%%%%%%%%%%%%%%%%%%%%%%%%%%%%%%%%%%
%%%%%%%%%%%%%%%%%%%%%%%%%%%%%%%%%%%%%%

\section{ Preliminaries} \label{Sec-prelim}

\subsection{Schr\"odinger representations of the Heisenberg group} We begin with recalling some basic representation theory of the Heisenberg group which provide the background for the Weyl transform and the twisted Bergman spaces. For more details on the material in  this section we refer to \cite{Book-Folland-phase-space, Book-Thangavelu-uncertainty}.  Let us denote by $\mathbb{H}^{n} =: \C^n \times \R $ the $(2n+1)$ dimensional Heisenberg group with the group law
$$(z,t)(w,s) = \left( z+w, t+s+ \frac{1}{2} \Im (z \cdot \bar{w}) \right)$$ 
for $(z,t), (w,s) \in \C^n \times \R$. 
It turns out that $\mathbb H^n$ is a unimodular group whose Haar measure is just the Lebesgue measure $dz \, dt$ of $\mathbb{C}^{n} \times \mathbb{R}$.   In order to define the Fourier transform on $\mathbb{H}^{n} $, and  the closely related Weyl transform,  we need to recall the unitary dual of the Heisenberg group.\\

The representation theory of the Heisenberg group is very easy to describe. There are two families of irreducible unitary representations for $  \He^n $ of which only the infinite dimensional ones contribute to the Plancherel measure. For every $\lambda \in \R^\ast= \mathbb{R} \setminus \left\lbrace 0\right\rbrace$, consider the map  $\pi_{\lambda} : \mathbb H^n \to U(L^2(\R^n)),$ the space of all unitary operators on $ L^2(\R^n),$ defined by the prescription
\begin{align} \label{def:schodinger-rep}
\pi_{\lambda}(z,t) \phi(\xi) = e^{i \lambda t} e^{i \lambda (x\cdot\xi+ \frac{1}{2}x\cdot y)} \phi(\xi + y), 
\end{align}
where $\phi\in L^{2}(\mathbb{R}^{n})$ and $z=x+ iy$. 
It is known that $\pi_{\lambda}$ is an irreducible unitary representation of $\mathbb{H}^{n}$, called the Schr\"odinger representation. Thanks to a theorem of Stone-von Neumann (see \cite{Book-Folland-phase-space}), it is known that any irreducible unitary representation of $\mathbb{H}^{n}$ which is nontrivial at the centre of $ \He^n$ is unitarily equivalent to exactly one of the $\pi_{\lambda}$'s.\\

Next we recall some facts concerning metaplectic representations that will be used in this work. Let $U(n)$ be the group of $n \times n$ unitary matrices with entries from complex field. It follows from the group law of $\mathbb H^n$ that for each $\sigma \in U(n)$, the map $(z,t) \to (\sigma z,t)$ is an automorphism of $\mathbb H^n$. With $\pi_{\lambda}$ denoting the Schr\"odinger representation as in \eqref{def:schodinger-rep}, one can easily verify that $(z,t) \to \pi_{\lambda} (\sigma z,t)$ is also an irreducible unitary representation of $\mathbb H^n$ which agrees with $\pi_{\lambda}$ at the centre $\C^{n} \times \{0\}$ of $\mathbb H^n$. Therefore, it follows from the theorem of Stone-von Neumann that there is a unitary operator $\mu_\lambda (\sigma)$ on $L^2(\R^n)$ such that 
\begin{align} \label{relation:metaplectic-rep-Schorodinger-rep}
\pi_{\lambda} (\sigma z, t) = \mu_\lambda (\sigma) \pi_\lambda (z,t) \mu_\lambda (\sigma)^*. 
\end{align} 
Without going into details, let us just mention here that it is possible to choose such an operator valued function $\mu_\lambda$ on $U(n)$ in such a way that it defines a unitary representation of the double cover of the symplectic group which is called the metaplectic representation. We  refer to \cite{Book-Folland-phase-space} for the details.

%%%%%%%%%%%%%%%%%%%%%%%%%%%
%%%%%%%%%%%%%%%%%%%%%%%%%%%
%%%%%%%%%%%%%%%%%%%%%%%%%%%
%%%%%%%%%%%%%%%%%%%%%%%%%%%
%%%%%%%%%%%%%%%%%%%%%%%%%%%

\subsection{Weyl transform, Weyl multipliers and  Weyl correspondence } \label{subsec-prelim-Weyl-tr-metaplectic-rep}

Using the representations $ \pi_\lambda $  we define the group Fourier transform of any $f \in L^1(\mathbb{H}^{n})$ as the operator valued function on $ \R^\ast$ given by
$$\widehat{f}(\lambda) = \int_{\mathbb{H}^{n}} f(z,t) \pi_{\lambda}(z,t) \, dz \, dt.$$ 
It can be shown that for $ f \in L^1 \cap L^2(\He^n) $ one has the Plancherel theorem which reads as 
$$  \int_{\mathbb H^n} |f(z,t)|^2 dz\, dt  = \int_{-\infty}^\infty \| \widehat{f}(\lambda)\|_{HS}^2\, d\mu(\lambda)$$ 
where $  d\mu(\lambda) = (2\pi)^{-n-1}  |\lambda|^n d\lambda$ is the Plancherel measure for $ \He^n.$ It is well known that the operator $ f \rightarrow \widehat{f} $ extends to the whole of $L^{2}(\He^{n})$ as a unitary map onto 
$ L^2(\R^\ast, \mathcal{S}_2, d\mu) $  where $ \mathcal{S}_2 $ is the Hilbert space of Hilbert-Schmidt operators on $L^{2}(\mathbb{R}^{n})$.\\

From the definition of $ \pi_\lambda$ we see that $ \pi_\lambda(z,t) = e^{i\lambda t}\, \pi_\lambda(z) $ where  we have set $   \pi_\lambda(z) = \pi_\lambda(z,0)$ and hence 
$$ \widehat{f}(\lambda) = \int_{\C^n} f^\lambda(z) \, \pi_\lambda(z)\, dz .$$  Here by $ f^\lambda(z) $ we mean the inverse Fourier transform of $ f(z,t) $ in the last variable:
$$ f^\lambda(z) = \int_{-\infty}^\infty f(z,t) e^{i\lambda t}\, dt.$$
This motivates us to define Weyl transform of a function $ g $ on $\mathbb{C}^n$ as follows: 
\begin{align*} %\label{def-Weyl}
\pi_{\lambda}(g) := \int_{\mathbb{C}^{n}} g(z)\pi_{\lambda} (z)\, dz. 
\end{align*}
Observe that $ \widehat{f}(\lambda) = \pi_\lambda(f^\lambda) $ for any $ f $ on $ \He^n.$ The Plancherel theorem for the Weyl transform then reads as 
$$ \int_{\C^n} |g(z)|^2 \, dz = (2\pi)^{n} |\lambda|^{-n} \|\pi_\lambda(g)\|_{HS}^2.$$ 
In fact, the Plancherel theorem for the Fourier transform on $ \He^n$ is proved by establishing the above identity.\\

The convolution of two functions  $f, g \in L^1(\mathbb{H}^{n})$ is defined in the usual way by
$$f \ast g (z,t)=\int_{\mathbb{H}^{n}} f\left((z,t)(w,s)^{-1}\right) g(w,s) \, dw \, ds.$$
It is then easy to check that $ \widehat{f \ast g}(\lambda) = \widehat{f}(\lambda)\,  \widehat{g}(\lambda).$ In a similar way, there is a convolution structure $\ast_{\lambda}$ on $ \C^n,$ called the $\lambda$-twisted convolution,  with respect to which the Weyl transform satisfies the relation $$ \pi_{\lambda}(F \ast_{\lambda} G) = \pi_{\lambda}(F) \pi_{\lambda}(G).$$ 
The $\lambda$-twisted convolution $\ast_{\lambda}$ between two functions $F$ and $G$ on $\mathbb{C}^n $ is defined by 
$$F \ast_{\lambda} G (z) = \int_{\mathbb{C}^{n}} F(z-w) G(w) e^{i \frac{\lambda}{2} \Im (z \cdot \bar{w})} \, dw.$$
It is easy to see that the convolution on $ \He^n $ and the $ \lambda$-twisted convolution  on $ \C^n $ are related by  $(f \ast g)^{\lambda} = f^{\lambda} \ast_{\lambda} g^{\lambda}.$ \\

For ready reference we record some easily verifiable properties of   the unitary operators $ \pi_\lambda(x,u) = \pi_\lambda(x+iu,0) .$ First of all we have
$$ \pi_\lambda(a,b)\pi_\lambda(x,u) = \pi_\lambda(x+a,u+b) e^{-i\frac{\lambda}{2}(u\cdot a- x\cdot b)}.$$
Recalling the definition of $ \pi_\lambda(f) $ another easy calculation shows that
$$ \pi_\lambda(a,b) \pi_\lambda(f) = \int_{\R^{2n}}  f(x,u) \pi_\lambda(x+a,u+b) e^{-i\frac{\lambda}{2}(u\cdot a- x\cdot b)} \, dx \,du .$$
In other words, the $\lambda$-twisted translation $ \tau_\lambda(a,b) $ defined by
\begin{align} \label{def:twisted-translation} 
\tau_\lambda(a,b)f(x,u) = f(x-a,u-b)\,e^{-i\frac{\lambda}{2}(u\cdot a- x\cdot b)} 
\end{align}
satisfies the relation $ \pi_\lambda(a,b) \pi_\lambda(f) = \pi_\lambda( \tau_\lambda(a,b)f) .$ This translates into the property
\begin{align} \label{trans-con}  \tau_\lambda(a,b) ( f\ast_\lambda g) = \tau_\lambda(a,b)f \ast_\lambda g 
\end{align}
for  the $\lambda$-twisted convolution of $ f $ with $ g .$ \\

Recall that bounded linear operators on $ L^2(\R^n) $ that commute  with translations on $ \R^n $ are of the form $ Tf = f \ast k $ for a distribution $k$ whose Fourier transform $ m = \widehat{k} \in L^\infty(\R^n)$. Thus $ \widehat{Tf}(\xi) = m(\xi) \widehat{f}(\xi) $ and for this reason the notation $ T_m $ is used for $ T.$ Such operators and the function $ m $  are called  Fourier multipliers. In a similar vein bounded linear operators $ T $ on $ L^2(\C^n) $ commuting with the $\lambda$-twisted translations $ \tau_\lambda(a,b)$ are of the form $ Tf = f \ast_\lambda k $ where $ k $ is a distribution, see \cite{MauceriWeylTransformJFA80}. By taking the Weyl transform we note that $ \pi_\lambda(Tf) = \pi_\lambda(f) \pi_\lambda(k) $ and for this reason the linear operator $ M =\pi_\lambda(k) $ is called the right Weyl multiplier. The notation $ T_M $ is used for the operator $ T $ and it is not difficult to show that $ T_M $ is bounded on $ L^2(\C^n) $ if and only if $ M $ is a bounded linear operator on $ L^2(\R^n).$ Weyl multipliers play an important  role in our study of bounded linear operators on twisted Fock spaces.\\

In what follows we also require the notion of the Weyl correspondence which is intimately related to the Weyl transform. For this, let us first recall that the symplectic Fourier transform $\mathcal{F}_\lambda$ of a function $ f \in L^1(\C^n) $ is defined by
\begin{align} \label{def:symplectic-Fourier-transform} 
\mathcal{F}_\lambda f(z) = (2\pi)^{-n} \int_{\C^n} f(z-w) e^{i \frac{\lambda}{2} \Im(z \cdot \bar{w})} \, dw  = (2\pi)^{-n} f \ast_\lambda 1(z). 
\end{align} 
We note the relation between the symplectic Fourier transform and the Euclidean Fourier transform on $\R^{2n}$, which is $ \mathcal{F}_\lambda f(z) = \widehat{f}(-\frac{i}{2}\lambda z).$ 
One can extend the definition of the symplectic Fourier transform to the space of tempered distributions. With that, given a tempered distribution $ f $, let us define its Weyl correspondence $W_\lambda(f)$ by 
\begin{align} \label{def:Weyl-correspondence} 
W_\lambda(f) = \pi_\lambda(\mathcal{F}_\lambda f). 
\end{align} 
Note that when $f = P$ is a polynomial, $\mathcal{F}_\lambda P $ is nothing but a finite sum of derivatives of the Dirac delta, from which it follows that $ W_\lambda(P) $ is a differential operator. Building on this fact, Geller \cite{Geller-spherical-harmonics-weyl-transform} described  operator analogues of spherical harmonics. We shall recall this later (see \Cref{subsec-prelim-operator-sph-harmonics}). 

%%%%%%%%%%%%%%%%%%%%%%%%%%%
%%%%%%%%%%%%%%%%%%%%%%%%%%%
%%%%%%%%%%%%%%%%%%%%%%%%%%%
%%%%%%%%%%%%%%%%%%%%%%%%%%%
%%%%%%%%%%%%%%%%%%%%%%%%%%%

\subsection{Sublaplacian, Hermite and special Hermite operators}  The role of the Laplacian for the Heisenberg group is played by the so called sublaplacian which is defined  by $ \mathcal{L} = - \sum_{j=1}^n (X_j^2 +Y_j^2).$  Here  $X_j, Y_j $ are the left invariant vector fields on $ \He^n $ given explicitly by
$$ X_j = \frac{\partial}{\partial x_j} +\frac{1}{2} y_j \frac{\partial}{\partial t}, \quad Y_j = \frac{\partial}{\partial y_j} - \frac{1}{2} x_j \frac{\partial}{\partial t}.$$
The sublaplacian $ \mathcal{L} $ on the Heisenberg group gives rise to a family of operators $ L_\lambda,$ indexed by non-zero reals $ \lambda.$  The relation between $ \mathcal{L} $ and $ L_\lambda $ is given by $ \mathcal{L}(e^{i\lambda t} g(z)) = e^{i\lambda t} L_\lambda g(z).$  We also have the relation $ (\mathcal{L}f)^\lambda(z) = L_\lambda f^\lambda(z).$ These operators $ L_\lambda $ are known as special Hermite operators (also called  twisted Laplacians) and they generate  diffusion semigroups with explicit heat kernels given by 
\begin{align} \label{def:heat-kernel-special-Hermite}
p_t^\lambda(y,v) = (4\pi)^{-n} \left( \frac{\lambda}{ \sinh \lambda t}\right)^n e^{-\frac{1}{4}\lambda (\coth \lambda t)(|y|^2+|v|^2)}. 
\end{align}
For more about the special Hermite operators and the semigroup generated by them we refer to the monograph \cite{Thangavelu-Hermite-Laguerre-Expansions-Book}. \\

The sublaplacian $ \mathcal{L} $ on $ \He^n $ is also intimately connected to the Hermite operators $ H(\lambda) = -\Delta+\lambda^2 |x|^2 $ for $ \lambda \in \R^\ast.$ We have the relation $ \widehat{\mathcal{L}f}(\lambda) = \widehat{f}(\lambda)\,H(\lambda).$ In view of this the spectral theory of $ \mathcal{L} $ depends quite a lot on the spectral decomposition of $ H(\lambda).$ The latter is provided by the Hermite functions. For each $ \alpha \in \mathbb{N}^n $ we let $ \Phi_\alpha $ stand for the normalised Hermite function 
$$ \Phi_\alpha(x) = c_\alpha \, e^{\frac{1}{2}|x|^2} \partial^\alpha e^{-|x|^2} = c_\alpha\, H_\alpha(x) e^{-\frac{1}{2}|x|^2} .$$
Here $ c_\alpha $ is chosen so that  $ \| \Phi_\alpha \|_2 =1.$ The Hermite functions are eigenfunctions of $ H= H(1) $ with eigenvalues $ (2|\alpha|+n) $ and they form an orthonormal basis for $ L^2(\R^n).$ For each $ \lambda \in \R^\ast $ we let $ \Phi_\alpha^\lambda(x) = |\lambda|^{n/4} \Phi_\alpha(|\lambda|^{1/2} x) $ so that $ H(\lambda)  \Phi_\alpha^\lambda(x) = (2|\alpha|+n)|\lambda| \,  \Phi_\alpha^\lambda(x).$ For all the properties of the Hermite functions that we use in this work we refer to the monographs \cite{Thangavelu-Hermite-Laguerre-Expansions-Book, Book-Thangavelu-uncertainty}.\\

The Hermite functions $ \Phi_\alpha^\lambda, \, \alpha \in \mathbb{N}^n $ form an orthonormal basis for $ L^2(\R^n).$ Given $ f \in L^2(\R^n) $ we let $ P_k (\lambda) $ stand for the orthogonal projection of $ L^2(\R^n) $ onto the eigenspace $ E_k^\lambda$ which is just the span of $ \Phi_\alpha^\lambda, \, |\alpha|=k$. More explicitly, for any $ k \in \mathbb{N}$, we have $ P_k (\lambda) f = \sum_{|\alpha|=k} (f, \Phi_\alpha^\lambda) \Phi_\alpha^\lambda.$ The spectral decomposition of $ H(\lambda) $ then reads as 
$$ H(\lambda) = \sum_{k=0}^\infty (2k+n)|\lambda|\, P_k(\lambda).$$
The operator $ H(\lambda) $ has the following factorisation:
$$ H(\lambda) = \frac{1}{2} \sum_{j=1}^n \left(A_j(\lambda) A_j^\ast(\lambda) + A_j^\ast(\lambda) A_j(\lambda) \right) $$
where $ A_j^\ast(\lambda) = -\frac{\partial}{\partial x_j}+|\lambda| x_j $ and $A_j(\lambda) = \frac{\partial}{\partial x_j}+|\lambda| x_j $ are the creation and annihilation operators of quantum mechanics. The action of these operators on $ \Phi_\alpha^\lambda $ are explicitly known and we make use of them in our proofs later.\\

Among other properties of the metaplectic representation $\mu_\lambda$ satisfying \eqref{relation:metaplectic-rep-Schorodinger-rep}, we make use the following one. It is known that for any $\sigma \in U(n)$, the operator $\mu_\lambda(\sigma)$ leaves the space $ E_k^\lambda $ invariant. Thus $ \mu_\lambda(\sigma) $ commutes with $ P_k(\lambda) $ hence also with operators of the form $ m(H(\lambda)) $ defined using spectral theorem.

%%%%%%%%%%%%%%%%%%%%%%%%%%%
%%%%%%%%%%%%%%%%%%%%%%%%%%%
%%%%%%%%%%%%%%%%%%%%%%%%%%%
%%%%%%%%%%%%%%%%%%%%%%%%%%%
%%%%%%%%%%%%%%%%%%%%%%%%%%%

\subsection{Spherical harmonics and their operator analogues} \label{subsec-prelim-operator-sph-harmonics}
We begin by recalling some relevant results from the theory of bigraded spherical harmonics and their operator analogues. This subsection is actually a reproduction of Section $2.5$ of \cite{Basak-Garg-Thangavelu-Weyl}. But, we prefer to write the details here as we need these tools and facts in Section \ref{Sec-proof-main-results}. \\

The unitary group $K=U(n)$  acts on any function space on the unit sphere $S^{2n-1}$. Let $ M$ be the subgroup of $K$ that fixes the coordinate vector $e_{1}=(1,0,\ldots,0)$. Since $K$ acts transitively on $S^{2n-1}$, we can identify $S^{2n-1}$ with $K/M$ via the map $\omega \rightarrow u M $ if $ \omega = u\cdot e_{1}$. The natural action of $K$ on $L^2(S^{2n-1})$ can be decomposed in terms of irreducible unitary representations having $ M$-fixed vectors known as class one representations. For each pair $(p,q) \in \mathbb{N}^2$, let $\mathcal{P}_{p,q}$ be the set of all polynomials on $\mathbb{C}^n$ which are of the form 
$$ P(z) = \sum_{|\alpha|=a} \sum_{|\beta|=b} a_{\alpha\beta} z^{\alpha} \bar{z}^{\beta}.$$ 
Each $P \in \mathcal{P}_{p,q}$ satisfies the homogeneity condition $P(s z) = s^{p} \bar{s}^{q} P(z)$ for all $s \in \mathbb{C} \setminus \{0\}$. Let $\Delta = 4 \sum_{j=1}^{n}\frac{\partial^{2}}{\partial z_{j}\partial\overline{z}_{j}}$ be the Laplacian on $\mathbb{C}^{n}$. Define $\mathcal{H}_{p,q} := \left\lbrace P\in\mathcal{P}_{p,q}: \Delta P=0 \right\rbrace$. The elements of $\mathcal{H}_{p,q}$ are called  bigraded solid harmonics. It is known that the representations $\delta_{p,q}$ of $U(n)$ defined on $\mathcal{H}_{p,q}$ by $\delta_{p,q}(\sigma)P(z) = P(\sigma^{-1} z)$ are irreducible unitary representations and exhaust all class one irreducible unitary representations of $ K,$ upto unitary equivalence. We denote this class of representations by $\widehat{K_0}$. 
For each $\delta=\delta_{p,q},$ we let $d(\delta)$ denote the dimension of $\mathcal{H}_{p,q}$ and  $\chi_\delta$  the character associated to $\delta$. \\

The space $\mathcal{H}_{p,q}$  is made into a Hilbert space by equipping it with the inner product
$$(f,g)_{\mathcal{H}_{p,q}} = \frac{2^{-(n+p+q-1)}}{\Gamma(n+p+q)} \int_{\mathbb{C}^n} f(z) \overline{g(z)} e^{-\frac{1}{2} |z|^2} \, dz.$$ 
We fix an orthonormal basis $\{P^{\delta}_{j}$ : $1\leq j\leq d(\delta)\}$ for  $\mathcal{H}_{p,q}.$  Then by defining the spherical harmonics $Y_j^\delta$ by the relation $ P_j^\delta(z) = |z|^{p+q} \,Y_j^\delta(\omega),$  for $z= |z| \, \omega$, the collection
$$\left\lbrace Y^{\delta}_{j}  : \delta \in \widehat{K_0}, 1\leq j\leq d(\delta) \right\rbrace $$ 
becomes an orthonormal basis for $L^{2}(S^{2n-1})$. Given a measurable function $ f $ on $ \C^n $ which  has a well defined restriction on every sphere $S_R = \{ z: |z| =R \}$, we have the spherical harmonic expansion 
\begin{equation} \label{harmonic-exp-1} 
f(R\omega) = \sum_\delta \sum_{j=1}^{d(\delta)} (f_R,Y_j^\delta)_{L^2(S^{2n-1})}  Y_j^\delta(\omega),
\end{equation}
where $ f_R(\omega) := f(R\omega)$ for $\omega \in S^{2n-1}.$ 
If $ z = R \omega, $ we can rewrite the above expansion in the following form. For each $R > 0$, let $ \sigma_R $ be the normalised surface measure on the sphere $S_R = \{ z: |z|=R \}$ defined by 
$$\int_{S_R} f(z) d\sigma_R = \int_{S^{2n-1}} f(R\omega) d\sigma.$$
Denoting the inner product in $ L^2(S_R,d\sigma_R) $ by $ (f,g)_R$, we can rewrite \eqref{harmonic-exp-1} as
\begin{align*} 
%\label{harmonic-2}
f(z) = \sum_\delta \sum_{j=1}^{d(\delta)} R^{-(p+q)} (f,P_j^\delta)_R   \,\, R^{-(p+q)} P_j^\delta(z), \quad \textup{for } z \in S_R. 
\end{align*}

There is an operator analogue of the above mentioned spherical harmonics which we briefly recall below. For details, we refer to \cite{Geller-spherical-harmonics-weyl-transform, Thangavelu-Poisson-trans-Heisenberg}, \cite[Section $2.7$]{Book-Thangavelu-uncertainty} and  \cite[Section 2.5]{Basak-Garg-Thangavelu-Weyl}. For each $k \in \mathbb{N}$, consider the the following sub-collection of $\widehat{K_0}$: 
$$\widehat{K}(k) = \left\lbrace \delta_{p,q}\in \widehat{K_0}: 0 \leq p \leq k, q \in \mathbb{N}\right\rbrace. $$ 
Let $E_{k}^\lambda$ be the finite dimensional subspace defined in the previous subsection. Let $\mathcal{O} (E_{k}^\lambda)$ be the space of all bounded linear operators $T : E_{k}^\lambda \to L^{2}(\mathbb{R}^{n})$. Then we can make $\mathcal{O} (E_{k}^\lambda)$ into a Hilbert space by defining the following inner product: 
\begin{align*} 
%\label{def-O(E_k)-inner-prod}
\left(T, S\right)_{k} = \frac{k!(n-1)!}{(k+n-1)!} \sum_{|\alpha|=k} \left(T \Phi_{\alpha}^\lambda, S \Phi_{\alpha}^\lambda \right).
\end{align*}
Recall that we have fixed an orthonormal basis $\{P^{\delta}_{j}$ : $1\leq j\leq d(\delta)\}$ of $\mathcal{H}_{p,q}$ and the Weyl correspondence $W_\lambda(f)$ of a tempered distribution $ f $ is defined by the equation in \eqref{def:Weyl-correspondence}. In an impressive work \cite{Geller-spherical-harmonics-weyl-transform} Geller proved that the Weyl correspondences of $ P_j^\delta,$ 
$$ \{W_\lambda (P^{\delta}_{j}) : \delta\in \widehat{K}(k), 1\leq j \leq d(\delta) \} $$ 
form an orthogonal system in $\mathcal{O}(E_{k}^\lambda)$ and that every operator $T \in \mathcal{O}(E_{k}^\lambda)$ has the expansion 
\begin{align*} 
%\label{Oper-Spherical-Harmonic}
T = \sum_{\delta \in \widehat{K}(k)} \sum_{j=1}^{d(\delta)} (C_\delta ((2k+n) |\lambda|))^{-2} \left(T, W_\lambda (P^{\delta}_{j}) \right)_{k} W_\lambda (P^{\delta}_{j})
\end{align*}
where $(C_\delta ((2k+n) |\lambda|))^2 = (W_\lambda (P_j^\delta), W_\lambda (P_j^\delta)_k.$ These constants can be computed  explicitly, see \cite{Geller-spherical-harmonics-weyl-transform} :
$$(C_\delta ((2k+n) |\lambda|))^2 = c_\lambda \, 4^{p+q} \, 2^{n+p+q-1} \,  \frac{\Gamma(k+n+q)}{\Gamma(k-p+1)} \,  \frac{\Gamma(k+1) \, \Gamma(n)}{\Gamma(k+n)}.$$
It follows that $\{ (C_\delta ((2k+n) |\lambda|))^{-1} W_\lambda (P^{\delta}_{j}) : \delta\in \widehat{K}(k), 1\leq j \leq d(\delta)\}$ forms an orthonormal basis for $ \mathcal{O}(E_k^\lambda).$

For the convenience of the readers (and also for later use) let us record the above result of Geller in the following form.  For each $\delta \in \widehat{K}(k)$ and $ 1 \leq j \leq d(\delta) $ we define
$$ S_{j,k}^\delta (\lambda) = (C_\delta ((2k+n) |\lambda|))^{-1} W_\lambda (P_j^\delta) P_k (\lambda).$$
We note that $ S_{j,k}^\delta (\lambda)$ are Hilbert-Schmidt operators on $ L^2(\R^n) $ with unit norm. We have

\begin{thm}[Geller] The collection $ \{ S_{j,k}^\delta (\lambda) : k \in \mathbb{N}, \delta \in \widehat{K}(k), 1 \leq j \leq d(\delta) \}$ is an orthonormal basis for the Hilbert space $ \mathcal{S}_2 $ of Hilbert-Schmidt operators on $ L^2(\R^n) $ equipped with the inner product $ (T,S) = tr(S^\ast T).$ Moreover, for any Hilbert-Schmidt operator $ T $ on $ L^2(\R^n) $ we have
\begin{align*}
T = \sum_{k=0}^\infty \sum_{\delta \in \widehat{K}(k)} \sum_{j=1}^{d(\delta)} ( T, S_{j,k}^\delta (\lambda)) \, S_{j,k}^\delta (\lambda) 
\end{align*} 
where the series converges in $ \mathcal{S}_2 $ and we have the identity 
\begin{align*} 
\|T\|_{HS}^2 = \sum_{k=0}^\infty \sum_{\delta \in \widehat{K}(k)} \sum_{j=1}^{d(\delta)} \left| ( T, S_{j,k}^\delta (\lambda)) \right|^2. 
\end{align*} 
\end{thm} 

It is clear that the collection $ S_{j,k}^\delta (\lambda)$ is an orthonormal set in view of the definition of the inner product on $ \mathcal{S}_2 $ and the orthogonality properties of $ W_\lambda(P_j^\delta).$ The rest of the theorem follows from the fact that $ \|T\|_{HS}^2 = \sum_{k=0}^\infty \| T P_k (\lambda) \|_{HS}^2 $ and Geller's result, see \cite{Geller-spherical-harmonics-weyl-transform}. 

\medskip 
For  each fixed $\delta = \delta_{p,q} $ and $1 \leq j \leq d(\delta),$ we  define the operator
\begin{align} \label{Oper-Spherical-Harmonic-M-j-delta} 
T_j^\delta (\lambda) := \sum_{k=p}^{\infty} (C_\delta ((2k+n) |\lambda|))^{-2} \left(T, W_\lambda (P^{\delta}_{j})\right)_{k} P_k (\lambda). 
\end{align} 
For any $ T \in B(L^2(\R^n)) $ the above series in \eqref{Oper-Spherical-Harmonic-M-j-delta} converges in the strong operator topology.  Indeed, as $ f = \sum_{k=0}^\infty P_k (\lambda) f $ for any $ f \in L^2(\R^n) $ we only need to check that the sequence  $ (C_\delta ((2k+n) |\lambda|))^{-2} \left(T, W_\lambda (P^{\delta}_{j})\right)_{k}$ is bounded. But this is easy to see: by the definition
$$ (T, W_\lambda (P_j^\delta))_k = \frac{\Gamma(k+1) \Gamma(n)}{\Gamma(k+n)} \sum_{|\alpha|=k} (T \Phi_\alpha^\lambda, W_\lambda (P_j^\delta) \Phi_\alpha^\lambda). $$ 
Applying Cauchy-Schwarz inequality and recalling the definition of $ (C_{\delta., \lambda} (2k+n))^{2} $ we obtain 
$$ \left| \left( T, W_\lambda (P_j^\delta) \right)_k \right|^2 \leq \|T\|^2 (C_\delta ((2k+n) |\lambda|))^{2} $$
after making use of the fact that the dimension of $ E_k^\lambda $ is $ \frac{\Gamma(k+n)}{\Gamma(k+1) \Gamma(n)}.$  Thus we have the estimate
$$ (C_\delta ((2k+n) |\lambda|))^{-2} \left| \left( T, W_\lambda (P_j^\delta) \right)_k \right| \leq \|T\| (C_\delta ((2k+n) |\lambda|))^{-1} $$
which is clearly bounded in view of Stirling's formula for the Gamma function. The operator norm of $ T_j^\delta (\lambda)$ is given by
$$ \| T_j^\delta (\lambda) \| =  \sup_{k \in \mathbb N} \left\{ (C_\delta ((2k+n) |\lambda|))^{-2} \left| \left( T, W_\lambda (P_j^\delta) \right)_k \right| \right\} < \infty.$$

%%%%%%%%%%%%%%%%%%%%%%%%%%%%%
%%%%%%%%%%%%%%%%%%%%%%%%%%%%%
%%%%%%%%%%%%%%%%%%%%%%%%%%%%%
%%%%%%%%%%%%%%%%%%%%%%%%%%%%%

\begin{rem}  In view of the above discussion we have the formal expansion
$$  T = \sum_{\delta \in \widehat{K_0}} \sum_{j=1}^{d(\delta)} W_\lambda (P_j^\delta) T_j^\delta (\lambda) $$ 
for any $ T \in B(L^2(\R^n)).$ For any $ k \in \mathbb{N}$ we get back the convergent expansion 
$$ T P_k (\lambda) = \sum_{\delta \in \widehat{K_0}} \sum_{j=1}^{d(\delta)} (T,S_{j,k}^\delta (\lambda)) S_{j,k}^\delta (\lambda).$$ 
\end{rem}

%%%%%%%%%%%%%%%%%%%%%%%%%%%
%%%%%%%%%%%%%%%%%%%%%%%%%%%
%%%%%%%%%%%%%%%%%%%%%%%%%%%
%%%%%%%%%%%%%%%%%%%%%%%%%%%
%%%%%%%%%%%%%%%%%%%%%%%%%%%

\subsection{Twisted Bergman and Fock spaces} \label{subsec-prelim-twisted-bergman-fock-spaces} 
We start with recalling the twisted Bergman spaces $ \mathcal{B}_t^\lambda(\C^{2n}).$ 
The twisted Segal-Bargmann transform $ B_{t,\lambda} $ is defined on $ L^2(\R^{2n}) $ by 
\begin{align*} 
%\label{def:twisted-Segal-Bargmann-transform} 
B_{t,\lambda} f(z,w) = f \ast_\lambda p_t^\lambda(z,w) = \int_{\R^{2n}} f(a,b) p_t^\lambda(z-a,w-b) e^{-\frac{i}{2}  \lambda (w\cdot a-z\cdot b)} \, da\, db 
\end{align*} 
where $ p_t^\lambda $ is the heat kernel associated to the special Hermite operator $ L_\lambda$ defined in \eqref{def:heat-kernel-special-Hermite}. The image of $ L^2(\R^{2n}) $ under $ B_{t,\lambda}$  is known to be a weighted Bergman space denoted by $\mathcal{B}_t^\lambda(\C^{2n}) .$ More precisely, this turns out to be the space consisting of all entire functions $ F $ on $ \C^{2n} $ which are square integrable with respect to the weight function 
\begin{align*} 
%\label{def:twisted-Segal-Bargmann-transform-weight-function} 
W_t^\lambda(z,w) = 4^n e^{\lambda(u \cdot y- v \cdot x)} p_{2t}^\lambda(2y,2v) = 4^n e^{\lambda \Im(z \cdot \bar{w})} p_{2t}^\lambda(2y,2v). 
\end{align*} 
Moreover, for every $ f \in L^2(\R^{2n}) $ we have the identity
$$ \int_{\C^{2n}} |B_{t,\lambda} f(z,w)|^2 \, W_t^\lambda(z,w) \, dz\, dw = \int_{\R^{2n}} |f(a,b)|^2 da\, db $$
which makes $ B_{t,\lambda} $ into a unitary operator. The reproducing kernel for the twisted Bergman space $\mathcal{B}_t^\lambda(\C^{2n})$ is given by 
\begin{align} \label{formula:reproducing-kernel-twisted-Bergman-space}
K_t^\lambda \left( (z,w), \overline{(a,b)} \right) = p_{2t}^\lambda( z - \bar{a}, w - \bar{b}) e^{-\frac{i}{2} \lambda (w \cdot \bar{a} - z \cdot \bar{b})}. 
\end{align}
This simply means that for any $ f \in L^2(\R^{2n}) $ we have 
$$B_{t,\lambda} f(z,w) = \int_{\C^{2n}} B_{t,\lambda} f(a,b) K_t^\lambda \left( (z,w), \overline{(a,b)} \right) W_t^\lambda(a,b) \, da \,db.$$
We refer the reader to \cite{KTX}  for more information on twisted Bergman spaces $\mathcal{B}_t^\lambda(\C^{2n})$.\\

For each $ t> 0 $ we define the twisted Fock space $ \mathcal{F}_t^\lambda(\C^{2n}) $ to be the space of entire functions on $ \C^{2n} $ for which
$$ \int_{\C^{2n}} |F(z,w)|^2 \, w_t^\lambda(z,w)\, dz\, dw  < \infty $$
where the weight function $ w_t^\lambda(z,w)$ is given by 
\begin{align} \label{def:twisted-Fock-space-weight-function} 
w_t^\lambda(z,w) = e^{\lambda \Im(z \cdot \bar{w})} e^{-\frac{1}{2}\lambda (\coth 2t\lambda)(|z|^2+|w|^2)} . 
\end{align} 
When $ t = 1/2$ we simply write $ \mathcal{F}^\lambda(\C^{2n}) $ instead of $ \mathcal{F}_{1/2}^\lambda(\C^{2n}). $
It is easy to see that $ F \in \mathcal{F}_t^\lambda(\C^{2n}) $ if and only if 
\begin{align*} 
%\label{relation:twisted-Bergman-and-Fock-spaces} 
F(z,w) = F_0(z,w)\, p_{2t}^\lambda(z,w)^{-1}, \quad F_0 \in \mathcal{B}_t^\lambda(\C^{2n}). 
\end{align*} 
Thus there is a one to one correspondence between $ \mathcal{B}_t^\lambda (\C^{2n}) $ and $\mathcal{F}_t^\lambda(\C^{2n}).$ By defining 
\begin{align*} 
%\label{def:twisted-Fock-transform} 
G_{t,\lambda}f(z,w) =  p_{2t}^\lambda(z,w)^{-1} \, B_{t,\lambda}f(z,w) =  p_{2t}^\lambda(z,w)^{-1}\, f \ast_\lambda p_t^\lambda(z,w) 
\end{align*} 
we see that $ G_{t,\lambda} : L^2(\R^{2n}) \rightarrow \mathcal{F}_t^\lambda(\C^{2n}) $ is a unitary operator.
In view of this, it is easy to see that the reproducing formula for the twisted Fock space reads as 
\begin{align} \label{formula:reproducing-twisted-Fock-space} 
F(z,w) = \int_{\C^{2n}} F(a,b)\, e^{\frac{1}{2}\lambda (\coth 2t\lambda)(z \cdot \bar{a}+w \cdot \bar{b})} e^{-\frac{i}{2} \lambda (w \cdot \bar{a}- z \cdot \bar{b})} w_t^\lambda(a,b) \, da\, db. 
\end{align} 
We observe that when $ \lambda = 0 $ the twisted Fock space reduces to the standard Fock space $ \mathcal{F}(\C^{2n}) $ and the reproducing kernel becomes $ e^{\frac{1}{2}(z \cdot \bar{a}+w \cdot \bar{b})}$ which is the reproducing kernel for the Fock space $ \mathcal{F}(\C^{2n}) .$

%%%%%%%%%%%%%%%%%%%%%%%%%%%
%%%%%%%%%%%%%%%%%%%%%%%%%%%
%%%%%%%%%%%%%%%%%%%%%%%%%%%
%%%%%%%%%%%%%%%%%%%%%%%%%%%
%%%%%%%%%%%%%%%%%%%%%%%%%%%

\subsection{The operator \texorpdfstring{$U_{t,\lambda}$}{}} \label{subsec-prelim-intertwining-operator} 
Observe that the weight function $ w_t^\lambda(z,w) $, given by \eqref{def:twisted-Fock-space-weight-function}, satisfies $ w_t^\lambda(iz,iw) = w_t^\lambda(z,w) $ and hence the map $ UF(z,w) = f(-iz,-iw) $ is unitary on the twisted Fock space $ \mathcal{F}_t^\lambda(\C^{2n}) $. As $ G_{t,\lambda} : L^2(\R^{2n}) \rightarrow \mathcal{F}_t^\lambda(\C^{2n}) $ is unitary, it follows that there is a unitary operator $ U_{t,\lambda} $ on $ L^2(\R^{2n}) $ such that $ U \circ G_{t,\lambda} = G_{t,\lambda} \circ U_{t,\lambda}.$ More explicitly,
\begin{align} \label{convolution-property-U-t-lambda}
U_{t,\lambda}f \ast_\lambda p_t^\lambda(z,w) = e^{-\frac{1}{2} \lambda (\coth 2t\lambda)(z^2+w^2)}\, f \ast_\lambda p_t^\lambda(-iz,-iw) 
\end{align} 
for any $ f \in L^2(\R^{2n}).$  When $ t = 1/2 $ we simply write $ U_\lambda $ instead of $ U_{1/2,\lambda}$ so that we have  the relation
\begin{align} \label{commuting-property-U-t-lambda}
U \circ G_{\lambda} = G_{\lambda} \circ U_{\lambda}. 
\end{align}
The operator $U_{t,\lambda}$, which takes $ p_t^\lambda $ into itself, has the following explicit description.

\begin{prop} \label{prop:intertwining-operator-U-t-lambda}
Let $ c_t(\lambda)= \frac{1}{2} \lambda (\coth t \lambda).$ For any $ f \in  L^2(\R^{2n}) $ we have  
\begin{align} \label{formula:explicit-expression-intertwining-operator-U-t-lambda}
U_{t,\lambda}f(x,u) = c_t(\lambda)^n \, \widehat{f}(c_t(\lambda)(x,u)) 
\end{align}
where $\widehat{f}$ is the Euclidean Fourier transform of $f$ on $\mathbb R^{2n}$.
\end{prop}
\begin{proof} It is enough to show that the operator $ T_{t,\lambda} f(x,u) = c_t(\lambda)^n \, \widehat{f}(c_t(\lambda)(x,u)) $ satisfies the same defining relation as $ U_{t,\lambda}.$
For the same, we shall show that for all $ (x,u) \in \R^{2n},$
$$ T_{t,\lambda}f \ast_\lambda p_t^\lambda(x,u) = e^{-\frac{1}{2} \lambda (\coth 2t\lambda)(x^2+u^2)}\, f \ast_\lambda p_t^\lambda(-ix,-iu),$$
which amounts to verifying the identity
\begin{align*}
& c_t(\lambda)^n\, \int_{\R^{2n}} \widehat{f}(c_{t,\lambda}(a^\prime,b^\prime))\, \tau_\lambda(a^\prime,b^\prime) p_t^\lambda(x,u) \, da^\prime\, db^\prime \\ 
& = e^{-\frac{1}{2} \lambda (\coth 2t\lambda)(x^2+u^2)}\, \int_{\R^{2n}} f(a,b) \tau_\lambda(a,b) p_t^\lambda(-ix,-iu) \, da \,db. 
\end{align*}
But then after throwing the Fourier transform in the integral on the left upon the function $\tau_\lambda(a^\prime,b^\prime) p_t^\lambda(x,u) $ we are left with verifying
\begin{align*}
& c_t(\lambda)^n\, (2\pi)^{-n} \,\int_{\R^{2n}} e^{- \frac{i}{2}  \lambda (\coth t\lambda)(a,b)\cdot (a^\prime,b^\prime)}  \tau_\lambda(a^\prime,b^\prime) p_t^\lambda(x,u) da^\prime\, db^\prime \\ 
& = e^{-\frac{1}{2} \lambda (\coth 2t\lambda)(x^2+u^2)} \, \tau_\lambda(a,b) p_t^\lambda(-ix,-iu). 
\end{align*}
As the kernel $ p_t^\lambda $ is known explicitly, the above can be verified by direct calculation using the identity $ 2 \coth 2t \lambda = (\coth t\lambda + \tanh t\lambda) .$
\end{proof}

\begin{rem} In view of the description \eqref{formula:explicit-expression-intertwining-operator-U-t-lambda} of $ U_{t,\lambda} $ we see that $ U_{t,\lambda}p_t^\lambda = p_t^\lambda.$ Thus $ p_t^\lambda $ is an eigenfunction of the operator $ U_{t,\lambda}.$ Other eigenfunctions of $ U_{t,\lambda} $ can be obtained by considering functions of the form $ f(z) = P(z) p_t^\lambda(z) $ where $ P $ is a bigraded solid harmonic. This follows from the Hecke-Bochner formula for the Fourier transform on $ \R^{2n}.$
\end{rem}

\begin{rem} It is possible to express $ U_{t,\lambda} $ in terms of the symplectic Fourier transform $\mathcal{F}_\lambda$ (see \eqref{def:symplectic-Fourier-transform} for the definition of $\mathcal{F}_\lambda$). In view of the relation $ \mathcal{F}_\lambda f(x,u) = \widehat{f} \left( \frac{1}{2} \lambda(u,-x) \right),$ we have 
$U_{t,\lambda}f(x,u) = c_t(\lambda)^n\, \mathcal{F}_\lambda(f) ((\coth t\lambda)(-u,x)).$ 
\end{rem}

%%%%%%%%%%%%%%%%%%%%%%%%%%%%%%%%%%%%%%
%%%%%%%%%%%%%%%%%%%%%%%%%%%%%%%%%%%%%%
%%%%%%%%%%%%%%%%%%%%%%%%%%%%%%%%%%%%%%
%%%%%%%%%%%%%%%%%%%%%%%%%%%%%%%%%%%%%%
%%%%%%%%%%%%%%%%%%%%%%%%%%%%%%%%%%%%%%
%%%%%%%%%%%%%%%%%%%%%%%%%%%%%%%%%%%%%%
%%%%%%%%%%%%%%%%%%%%%%%%%%%%%%%%%%%%%%
%%%%%%%%%%%%%%%%%%%%%%%%%%%%%%%%%%%%%%
%%%%%%%%%%%%%%%%%%%%%%%%%%%%%%%%%%%%%%
%%%%%%%%%%%%%%%%%%%%%%%%%%%%%%%%%%%%%%

\section{Weyl multipliers and twisted Fock spaces} \label{Sec-Weyl-multipliers-Fock-spaces} 
In this section we introduce the operators $ S_\varphi $ and study their boundedness properties on twisted Fock spaces $\mathcal{F}^\lambda(\C^{2n})$. This is achieved by conjugating them with the twisted Segal-Bargmann transform $ B_{\lambda} = B_{\frac{1}{2},\lambda} $ and relating them with Weyl multipliers.

%%%%%%%%%%%%%%%%%%%%%%%%%%%
%%%%%%%%%%%%%%%%%%%%%%%%%%%
%%%%%%%%%%%%%%%%%%%%%%%%%%%
%%%%%%%%%%%%%%%%%%%%%%%%%%%
%%%%%%%%%%%%%%%%%%%%%%%%%%%

\subsection{Boundedness of the operators \texorpdfstring{$ S_\varphi$}{}}
With notations as in \Cref{subsec-prelim-twisted-bergman-fock-spaces}, the reproducing formula for the twisted Bergman space $ \mathcal{B}_t^\lambda(\C^{2n}) $  suggests that we consider operators of the form
\begin{align} \label{def:convolution-operator-tiwsted-Bergman-space}
T_{\varphi_0}F(z,w) = \int_{\C^{2n}} F(a,b)  \varphi_0(z-\bar{a}, w-\bar{b}) e^{-\frac{i}{2} \lambda (w \cdot \bar{a}- z \cdot \bar{b})} W_t^\lambda(a,b) \, da \,db 
\end{align} 
where $ \varphi_0 = g \ast_\lambda p_t^\lambda $ is an element of the twisted Bergman space $ \mathcal{B}_t^\lambda(\C^{2n}) $. Note that for $ g = p_t^\lambda,$ we get $\varphi_0= p_{2t}^\lambda $ and $ T_{\varphi_0} = I.$
The operators $ T_{\varphi_0}$ described in \eqref{def:convolution-operator-tiwsted-Bergman-space} give rise to operators $S_\varphi$ on the twisted Fock space $ \mathcal{F}_t^\lambda(\C^{2n}) $, as described in \eqref{def:convolution-operator-tiwsted-fock}. For convenience, we rewrite the expression of $S_\varphi$ here. For any $ \varphi \in \mathcal{F}_t^\lambda(\C^{2n}) $ we consider
\begin{align*} 
%\label{def-repeated:convolution-operator-tiwsted-fock}
S_\varphi F(z,w) = \int_{\C^{2n}} F(a,b)  \varphi(z-\bar{a}, w-\bar{b}) e^{\frac{1}{2}\lambda (\coth 2t\lambda)(z \cdot \bar{a}+w \cdot \bar{b})} e^{-\frac{i}{2} \lambda (w \cdot \bar{a}- z \cdot \bar{b})} w_t^\lambda(a,b) \, da \,db. 
\end{align*} 
The relation between $ S_\varphi $ and $ T_{\varphi_0} $ is given by the easy to verify formula
\begin{align*}
S_\varphi F(z,w) = p_{2t}^\lambda(z,w)^{-1} \, T_{\varphi_0}F_0(z,w) 
\end{align*}
where $ F_0(z,w) = F(z,w) \, p_{2t}^\lambda(z,w) $ and $ \varphi_0(z,w) = \varphi(z,w) \, p_{2t}^\lambda(z,w).$
We are interested in finding conditions on $ \varphi$ so that $ S_\varphi $ is bounded on $ \mathcal{F}_t^\lambda(\C^{2n})$ or equivalently finding conditions on $ \varphi_0 $ so that $ T_{\varphi_0} $ is bounded on $ \mathcal{B}_t^\lambda(\C^{2n}).$\\

As earlier, when $ t = 1/2$ let us write $ B_\lambda $  instead of $ B_{t, \lambda} $ and $ \mathcal{B}^\lambda(\C^{2n}) $ instead of $ \mathcal{B}_t^\lambda(\C^{2n}).$ Then $ B_\lambda : L^2(\R^{2n}) \rightarrow \mathcal{B}^\lambda(\C^{2n}) $ is a unitary operator. An easy calculation shows that operators of the form $ T_{\varphi_0} $ commute with $ \tau_\lambda(a,b) $ for all $ (a,b) \in \R^{2n},$ where $\tau_\lambda (a,b)$ are the twisted translations defined in
\eqref{def:twisted-translation}. This in turn shows that the transferred operators $ B_\lambda^\ast \circ T_{\varphi_0} \circ B_\lambda $  on $ L^2(\R^{2n}) $ also commute with the twisted translations $ \tau_\lambda(a,b)$ for all $(a, b ) \in \R^{2n} .$  This is a consequence of the easily verifiable fact that $ B_\lambda^\ast \circ \tau_\lambda(a,b) \circ B_\lambda = \tau_\lambda(a,b).$ In fact, this follows from the relation \eqref{trans-con}. Consequently, the transferred operator $ B_\lambda^\ast \circ T_{\varphi_0} \circ B_\lambda $ turns out to be a right Weyl multiplier:
$$  (B_\lambda^\ast \circ T_{\varphi_0} \circ B_\lambda) f = T_M f, \quad  \pi_\lambda(T_Mf) = \pi_\lambda(f) M $$
for some linear operator $ M $ on $ L^2(\R^n),$ see \cite{MauceriWeylTransformJFA80}. Thus $ T_{\varphi_0} $ is bounded on $ \mathcal{B}^\lambda(\C^{2n}) $ if and only if the operator $ M $ associated to $ T_{\varphi_0}$ is a bounded linear operator on $ L^2(\R^n).$ \\

This analysis can be transferred to the Fock space setting. On $ \mathcal{F}^\lambda(\C^{2n}) $ consider the operators $ \rho_\lambda(a,b) $ described in \eqref{def:operators-rho-lambda} which we recall here for the reader's convenience:   
$$ \rho_\lambda(a,b)F(z,w) = e^{-i\frac{\lambda}{2}( w\cdot \bar{a}- z \cdot \bar{b})}\, e^{\frac{1}{2}\lambda (\coth \lambda)( z\cdot \bar{a}+w \cdot \bar{b})} \, F(z-a,w-b).$$
With this definition we observe that 
$$p_1^\lambda(a,b)\,  \rho_\lambda(a,b)F(z,w) = p_1^\lambda(z,w)^{-1} \tau_\lambda(a,b)( p_1^\lambda F)(z,w), $$ 
using which we  easily verify that the operators $ S_\varphi $ commute with $ \rho_\lambda(a,b) $ for all $ (a,b) \in \R^{2n}.$ 
Before moving on, let us show the following important fact which we had mentioned in the introduction.

\begin{prop} \label{prop:converse-statement-bdd-linear-op-twisted-fock-space-commuting-with-twisted-translation}
If a bounded linear operator $S$ on the twisted Fock space $ \mathcal{F}^\lambda(\C^{2n}) $ commutes with $ \rho_\lambda(a,b)$ for all $ (a, b)\in \R^{2n},$ then $S = S_\varphi$ for some $\varphi \in \mathcal{F}^\lambda(\C^{2n})$.
\end{prop}
\begin{proof}
Recall from \eqref{formula:reproducing-twisted-Fock-space} that for $ F \in \mathcal{F}^\lambda(\C^{2n})$ we have  
$$ F(z,w) = \int_{\C^{2n}} F(a,b) \, K_{(\bar{a},\bar{b})}(z,w) \, da\, db, $$ 
where the reproducing kernel $ K $ is given by
$$ K_{(\bar{a},\bar{b})}(z,w) = e^{\frac{1}{2}\lambda (\coth \lambda)(z \cdot \bar{a}+w \cdot \bar{b})} e^{-\frac{i}{2} \lambda (w \cdot \bar{a}- z \cdot \bar{b})}. $$
Since $S$ is a bounded linear operator on $ \mathcal{F}^\lambda(\C^{2n}) $, it follows that 
$$ SF(z,w) = \int_{\C^{2n}} F(a,b) \, K_S \left( (z,w), (\bar{a},\bar{b}) \right)  da\, db $$ 
where $K_S \left( (z,w), (\bar{a},\bar{b}) \right) = \left( S K_{(\bar{a},\bar{b})} \right) (z,w)$. 
As $ S $ commutes with $ \rho_\lambda(a',b') $ for all $ (a',b') \in \R^{2n}$, we have the identity
\begin{align*}
& e^{-i\frac{\lambda}{2}( w \cdot a'- z \cdot b')}\, e^{\frac{1}{2}\lambda (\coth \lambda)( z \cdot a' + w \cdot b')} \, SF(z-a',w-b') \\ 
& = \int_{\C^{2n}} e^{-i\frac{\lambda}{2}( b \cdot a' - a \cdot b')} \, e^{\frac{1}{2}\lambda (\coth \lambda)( a \cdot a' + b \cdot b')} \, F(a-a',b-b') \, K_S \left( (z,w), (\bar{a},\bar{b}) \right) da \, db 
\end{align*} 
for every $F \in \mathcal{F}^\lambda(\C^{2n})$ and $ (a',b') \in \R^{2n}$. This  implies that the kernel $ K_S $ satisfies
\begin{align}
& e^{-i\frac{\lambda}{2}( w \cdot a'- z \cdot b')}\, e^{\frac{1}{2}\lambda (\coth \lambda)( z \cdot a' + w \cdot b')} \, K_S \left( (z-a',w-b'), (\bar{a},\bar{b}) \right) \label{calc-relation-1} \\ 
\nonumber & = e^{-i\frac{\lambda}{2}( b \cdot a' - a \cdot b')} \, e^{\frac{1}{2}\lambda (\coth \lambda)( a \cdot a' + b \cdot b')} \, e^{\frac{1}{2}\lambda (\coth \lambda)( |a'|^2 + |b'|^2)} \, K_S \left( (z,w), (\bar{a}+a',\bar{b}+b') \right). 
\end{align} 

If we define $k_S \left( (z,w), (\bar{a},\bar{b}) \right) = e^{-\frac{1}{2}\lambda (\coth \lambda)(z \cdot \bar{a}+w \cdot \bar{b})} e^{\frac{i}{2} \lambda (w \cdot \bar{a}- z \cdot \bar{b})} K_S \left( (z,w), (\bar{a},\bar{b}) \right),$ then from \eqref{calc-relation-1} we get the relation
$$ k_S((z-a',w-b'), (a,b)) = k_S((z,w), (a+a',b+b')) $$ 
for all $z,w,a,b \in \C^n$ and $a',b' \in \R^n$, which in turn implies  that
$$ k_S((z,w), (a,b)) = k_S((z-a,w-b), (0,0)) .$$ 
Therefore, we can define $\varphi = k_S((z,w), (0,0))$ so that   
$$ K_S \left( (z,w), (\bar{a},\bar{b}) \right) = e^{\frac{1}{2}\lambda (\coth \lambda)(z \cdot \bar{a}+w \cdot \bar{b})} e^{-\frac{i}{2} \lambda (w \cdot \bar{a}- z \cdot \bar{b})} \varphi(z-\bar{a}, w-\bar{b}) .$$
This proves the claim that the operator $S$ is nothing but $S_\varphi$ that was defined in \eqref{def:convolution-operator-tiwsted-fock}. Finally, it follows from the reproducing formula that $\varphi = S1$ and hence $\varphi \in \mathcal{F}^\lambda(\C^{2n}).$
\end{proof}

\medskip 
Now, using the relation $ S_\varphi F(z,w) = p_1^\lambda(z,w)^{-1} T_{\varphi_0}(p_1^\lambda F) (z,w) $ we see that  each bounded linear operator $S_\varphi $ on $ \mathcal{F}^\lambda(\C^{2n}) $ corresponds to a Weyl multiplier $ T_M.$ In what follows we find a relation between the symbol $ \varphi $ and the multiplier $M.$ \\

As usual, let $ B(L^2(\R^n)) $ denote the Banach algebra of all bounded linear operators on $ L^2(\R^n).$ For each $ M \in B(L^2(\R^n)) $ we define
$$ G_\lambda(M)(z,w) = p_1^\lambda(z,w)^{-1}\, \tr\left( \pi_\lambda(-z,-w) e^{-\frac{1}{2} H(\lambda)} M e^{-\frac{1}{2}H(\lambda)} \right) .$$
As $ e^{-\frac{1}{2} H(\lambda)} M $ is Hilbert-Schmidt, we can find $ g \in L^2(\C^n) $ so that $ \pi_\lambda(g) = e^{-\frac{1}{2}H(\lambda)} M $ and hence in view of the inversion formula for the Weyl transform, $ G_\lambda(M)(z,w) $ is the holomorphic extension of $ p_1^\lambda(x,u)^{-1}\,g \ast_\lambda p_{1/2}^\lambda(x,u).$ Thus we see that $ G_\lambda $ takes $ B(L^2(\R^n)) $ into $ \mathcal{F}^\lambda(\C^{2n}).$ Moreover,
$$ \int_{\C^{2n}} |G_\lambda M(z,w)|^2 \, w_\lambda(z,w) \, dz\, dw = \left\| e^{-\frac{1}{2}H(\lambda)} M \right\|_{HS}^2.$$
We can now find a relation between the kernel $ \varphi $ and the operator $ M.$ 

\medskip 
\begin{proof}[\bf Proof of Theorem \ref{thm:convolution-operator-twisted-fock}] 
In view of the relation $ S_\varphi F(z,w) = p_{1}^\lambda(z,w)^{-1} T_{\varphi_0}F_0(z,w) $
with  $ F_0 = F\, p_{1}^\lambda $ and $ \varphi_0 = \varphi \, p_{1}^\lambda,$  it is enough to consider $ T_{\varphi_0}$ acting on the twisted Bergman space.
Let us apply the operator $ T_{\varphi_0} = B_\lambda \circ T_M \circ B_\lambda^\ast $ on the reproducing kernel associated to the space $ \mathcal{B}^\lambda(\C^{2n}).$ Since 
the reproducing kernel can be written as (put $t=1/2$ in \eqref{formula:reproducing-kernel-twisted-Bergman-space}) 
$$ p_1^\lambda(z-\bar{a}, w-\bar{b}) e^{-i\frac{\lambda}{2}(w \cdot \bar{b}- z\cdot \bar{a})} = \left(\tau_\lambda(\bar{a},\bar{b})p_{1/2}^\lambda \right)\ast_\lambda p_{1/2}^\lambda(z,w) $$
it follows that
$$ T_{\varphi_0} \left( \tau_\lambda(\bar{a},\bar{b})p_1^\lambda \right)(z,w) = B_\lambda \left( T_M \tau_\lambda(\bar{a},\bar{b}) p_{1/2}^\lambda \right)(z,w). $$
The right hand side can be calculated using the definition of $ B_\lambda $ in terms of the Weyl transform, namely
$$ B_\lambda f(z,w) = \tr \left( \pi_\lambda(-z,-w) \pi_\lambda(f) e^{-\frac{1}{2}H(\lambda)}\right).$$
Using the relation $ \pi_\lambda(a,b) \pi_\lambda(f) = \pi_\lambda( \tau_\lambda(a,b)f) $ (see the discussion around \eqref{def:twisted-translation}), by a simple calculation one can verify that
$$ B_\lambda \left( T_M \tau_\lambda(\bar{a},\bar{b}) p_{1/2}^\lambda \right)(z,w) = \tr \left( \pi_\lambda(\bar{a}, \bar{b}) \pi_\lambda(-z,-w) e^{-\frac{1}{2}H(\lambda)} M e^{-\frac{1}{2}H(\lambda)} \right). $$
This in turn shows that the kernel of the operator $ T_{\varphi_0} $ is given by
$$\tr \left( \pi_\lambda(-z+\bar{a}, -w+\bar{b}) e^{-tH(\lambda)} M e^{-tH(\lambda)} \right) \,e^{-\frac{i}{2} \lambda (w \cdot \bar{a}- z \cdot \bar{b})}.$$
Thus we obtain the relation 
$$ \varphi_0(z,w) = tr \left( \pi_\lambda(-z, -w) e^{-\frac{1}{2}H(\lambda)} M e^{-\frac{1}{2}H(\lambda)} \right).$$
As $ \varphi(z,w)  = p_{1}^\lambda(z,w)^{-1}\varphi_0(z,w) \, ,$ recalling the definition of $ G_\lambda $ from \eqref{def:Gauss-Bargmann-transform-tiwsted-fock}, we deduce that $ \varphi = G_\lambda(M),$ which completes the proof of Theorem \ref{thm:convolution-operator-twisted-fock}. 
\end{proof}

\begin{rem} Using the fact that $ T_M f = f \ast_\lambda \Lambda $ for a tempered distribution $ \Lambda $ satisfying $ \pi_\lambda(\Lambda) = M $ we see that $T_{\varphi_0} $ is bounded on $ \mathcal{B}^\lambda(\C^{2n}) $ if and only if $ \varphi_0(z,w) = p_{1/2}^\lambda \ast_\lambda \Lambda \ast_\lambda p_{1/2}^\lambda(z,w) $ where $ \pi_\lambda(\Lambda) $ is a bounded linear operator. 
\end{rem}

Let $ \mathcal{S}(\Gamma_t) $ be the space of all linear operators $ T $ on $ L^2(\R^n) $ such that $ e^{-tH(\lambda)} T$ is Hilbert-Schmidt equipped with the norm $ \|T\|_{(t)} = \left\| e^{-tH(\lambda)} T \right\|_{HS}.$ We define $ G_t^\lambda $ on $ \mathcal{S}(\Gamma_t) $ by the prescription
$$ G_t^\lambda T(z,w) = p_{2t}^\lambda(z,w)^{-1} \, \tr\big( \pi_\lambda(-z,-w) e^{-tH(\lambda)}T e^{-tH(\lambda)} \big) .$$
Then $ G_t^\lambda $ takes $ \mathcal{S}(\Gamma_t) $ onto the twisted Fock space $ \mathcal{F}_t^\lambda(\C^{2n}) .$ Indeed, as $ e^{-tH(\lambda)} T= \pi_\lambda(g) $ for some $ g \in L^2(\C^n) $ it follows that
$$\tr\big( \pi_\lambda(z,w) e^{-tH(\lambda)}T e^{-tH(\lambda)} \big) = g \ast_\lambda p_t^\lambda(-z,-w) $$ 
which belongs to $ \mathcal{B}_t^\lambda(\C^{2n}).$ Moreover,
$$ \int_{\C^{2n}} |G_t^\lambda T(z,w)|^2 \, w_t^\lambda(z,w) \, dz \, dw = \left\| e^{-tH(\lambda)} T \right\|_{HS}^2 = \| T\|_{(t)}^2.$$
It is also easy to see that $ G_t^\lambda $ is surjective: given $ F \in \mathcal{F}_t^\lambda(\C^{2n}) $ we have 
$$ F(-z,-w) \, p_{2t}^\lambda(z,w) = g \ast_\lambda p_t^\lambda(z,w) $$ for some $ g \in L^2(\C^n).$ Therefore, if we define $ T  = e^{tH(\lambda)}\,\pi_\lambda(g) , $ then it follows that $ F = G_t^\lambda (T).$ 

\begin{rem} \label{rem:adjoint-operator-G-lambda}
We observe that $ G_\lambda $ is the restriction of $ G_{1/2}^\lambda $ to $ B(L^2(\R^n)) \subset  \mathcal{S}(\Gamma_{1/2}).$  In other words, $ G_\lambda $ initially defined on $ B(L^2(\R^n)) $ has a natural extension to 
$ \mathcal{S}(\Gamma_{1/2}).$ Given $ F \in \mathcal{F}^\lambda(\C^{2n}) $ let $ F_0(z,w) = F(z,w) \, p_{1}^\lambda(z,w) $ be the corresponding member of $ \mathcal{B}^\lambda(\C^{2n}).$  We see that if we let $ T = e^{\frac{1}{2}H(\lambda)} \pi_\lambda(F_0) e^{\frac{1}{2}H(\lambda)} $
then $ G_\lambda(T) = F.$ Thus the adjoint of $ G_\lambda $ is given by the formula $ G_\lambda^\ast (F) = e^{\frac{1}{2}H(\lambda)} \pi_\lambda(F_0) e^{\frac{1}{2}H(\lambda)} .$
\end{rem}

\medskip 
\begin{proof}[\bf Proof of Corollary \ref{cor:convolution-operator-twisted-fock-via-Gauss-Bargmann-transform}] In view of the above, we can easily calculate $ G_\lambda^\ast(S_\varphi F).$   As in the proof of Theorem \ref{thm:convolution-operator-twisted-fock} we make use of the relation
$ S_\varphi F(z,w) = p_{1}^\lambda(z,w)^{-1} T_{\varphi_0}F_0(z,w) $
with  $ F_0 = F\, p_{1}^\lambda $ and $ \varphi_0 = \varphi \, p_{1}^\lambda.$ 
Therefore, in view of the Remark \ref{rem:adjoint-operator-G-lambda}, it follows that 
$$ G_\lambda^\ast (S_\varphi F) = e^{\frac{1}{2}H(\lambda)} \pi_\lambda(T_{\varphi_0}F_0) e^{\frac{1}{2}H(\lambda)} .$$
If $ F_0 = B_\lambda f,$ then as $ T_{\varphi_0} = B_\lambda \circ T_M \circ B_\lambda^\ast ,$ we have 
$$  e^{\frac{1}{2}H(\lambda)} \pi_\lambda(T_{\varphi_0}F_0) e^{\frac{1}{2}H(\lambda)} = e^{\frac{1}{2}H(\lambda)} \pi_\lambda(f) M.$$
 As $ G_\lambda^\ast(F) = e^{\frac{1}{2}H(\lambda)} \pi_\lambda(F_0) e^{\frac{1}{2}H(\lambda)} = e^{\frac{1}{2}H(\lambda)} \pi_\lambda(f) $ it follows that we have the relations
$$ G_\lambda^\ast(S_\varphi F) = (G_\lambda^\ast F) \circ  M, \quad S_\varphi F = G_\lambda\left( G_\lambda^\ast F \circ  M\right).$$
Since $ M = G_\lambda^\ast(\varphi) $, we have $ G_\lambda^\ast(S_\varphi F) = (G_\lambda^\ast F) \circ G_\lambda^\ast(\varphi),$ and this completes the proof of Corollary \ref{cor:convolution-operator-twisted-fock-via-Gauss-Bargmann-transform}. 
\end{proof}

%%%%%%%%%%%%%%%
%%%%%%%%%%%%%%%

\begin{rem} We can also consider $ G_t^\lambda $ as a map defined on a space of distributions. Let $ L^2(\R^{2n}, \Gamma_t) $ be the space of all tempered distributions $ \Lambda $ for which $ p_t^\lambda \ast_\lambda \Lambda  \in L^2(\R^{2n})$ and equip this space with the norm $ \| \Lambda\|_{(t)} = \| p_t^\lambda \ast_\lambda \Lambda \|_2.$ It is clear that $ L^2(\R^{2n}) \subset L^2(\R^{2n}, \Gamma_t) .$  On this space we define $ G_{t,\lambda} (\Lambda)(z,w) = 
p_{2t}^\lambda(z,w)^{-1}\, p_t^\lambda \ast_\lambda \Lambda \ast_\lambda p_t^\lambda(z,w)$ so that $ G_{t,\lambda} $ takes $ L^2(\R^{2n}, \Gamma_t) $
into $ \mathcal{F}_t^\lambda(\C^{2n}).$ The operator $ S_\varphi $ is bounded if and only if $ \varphi = G_{t,\lambda} (\Lambda) $ where $ \pi_\lambda(\Lambda) $ is a bounded linear operator.
\end{rem}

%%%%%%%%%%%%%%%%%%%%%%%%%%%
%%%%%%%%%%%%%%%%%%%%%%%%%%%
%%%%%%%%%%%%%%%%%%%%%%%%%%%
%%%%%%%%%%%%%%%%%%%%%%%%%%%
%%%%%%%%%%%%%%%%%%%%%%%%%%%

\subsection{An algebra of entire functions} 
\label{subsection-algebra-entire-functions} 
Let $ \mathcal{A^\lambda}(\C^{2n}) $ stand for the subspace of $ \mathcal{F}^\lambda(\C^{2n}) $ consisting of those  $ \varphi $ for which $ S_\varphi $ is bounded. In view of Theorem \ref{thm:convolution-operator-twisted-fock} we know that $ \varphi \in \mathcal{A^\lambda}(\C^{2n}) $ if and only if $ \varphi = G_\lambda M$ for some $ M \in B(L^2(\R^n)).$ Thus $\mathcal{A^\lambda}(\C^{2n}) $ is the image of $ B(L^2(\R^n)) $ under $ G_\lambda $ which turns out to be a Banach algebra under a suitable convolution. \\

The result of Corollary \ref{cor:convolution-operator-twisted-fock-via-Gauss-Bargmann-transform} suggests that for $ F, \varphi \in \mathcal{A^\lambda}(\C^{2n}) $ we define $ F \ast_\lambda \varphi = S_\varphi F:$ 
$$  F \ast_\lambda \varphi(z,w) = \int_{\C^{2n}}  F(a,b)  \varphi(z-\bar{a}, w-\bar{b}) e^{\frac{1}{2}\lambda (\coth \lambda)(z \cdot \bar{a}+w \cdot \bar{b})} e^{-\frac{i}{2} \lambda (w \cdot \bar{a}- z \cdot \bar{b})} w_t^\lambda(a,b) da \,db .$$ 
Observe that in view of Corollary \ref{cor:convolution-operator-twisted-fock-via-Gauss-Bargmann-transform} we have $ G_\lambda^\ast ( F \ast_\lambda \varphi) = G_\lambda^\ast (F) \, G_\lambda^\ast(\varphi) $ and hence it is clear that $ \mathcal{A^\lambda}(\C^{2n}) $ is a non-commutative  algebra under this convolution. Inside this algebra there is a commutative subalgebra which is easy to describe. For a bounded function $ m $ on $ \R $ let us consider the following bounded operator on $L^2(\R^n)$:  
$$ m(H(\lambda)) = \sum_{k=0}^\infty m((2k+n)|\lambda|) \, P_k(\lambda)$$ 
where $ P_k(\lambda) $ are the spectral projections associated to the Hermite operator $ H(\lambda).$ If we define
$$ \mathcal{A}_0^\lambda(\C^{2n}) = \{ \varphi \in \mathcal{A}^\lambda: G_\lambda^\ast(\varphi) =m(H(\lambda)) \} $$
it follows that $\mathcal{A}_0^\lambda(\C^{2n})$ is a commutative subalgebra. Moreover, every $ \varphi \in \mathcal{A}_0^\lambda(\C^{2n}) $ is given by the formula
$$ \varphi(z,w) = (2\pi)^{-n} |\lambda|^n \sum_{k=0}^\infty m((2k+n)|\lambda|) \, e^{-(2k+n)|\lambda|}\, \varphi_{k,\lambda}^{n-1}(z,w)$$
where $\varphi_{k,\lambda}^{n-1}(x,u)$ are dilated Laguerre functions of type $(n-1).$
It would be interesting to study some properties of this subalgebra such as Gelfand spectrum etc. We plan to return to such questions in a subsequent paper.

%%%%%%%%%%%%%%%%%%%%%%%%%%%%%%%%%%%%%%
%%%%%%%%%%%%%%%%%%%%%%%%%%%%%%%%%%%%%%
%%%%%%%%%%%%%%%%%%%%%%%%%%%%%%%%%%%%%%
%%%%%%%%%%%%%%%%%%%%%%%%%%%%%%%%%%%%%%
%%%%%%%%%%%%%%%%%%%%%%%%%%%%%%%%%%%%%%
%%%%%%%%%%%%%%%%%%%%%%%%%%%%%%%%%%%%%%
%%%%%%%%%%%%%%%%%%%%%%%%%%%%%%%%%%%%%%
%%%%%%%%%%%%%%%%%%%%%%%%%%%%%%%%%%%%%%
%%%%%%%%%%%%%%%%%%%%%%%%%%%%%%%%%%%%%%
%%%%%%%%%%%%%%%%%%%%%%%%%%%%%%%%%%%%%%

\section{The uncertainty principle} \label{Sec-proof-main-results}
In this section we present a proof of Theorem \ref{thm:uncertainty-twisted-fock-spaces}. Assuming that both $ S_\varphi$ and $ \widetilde{S}_\varphi $ are bounded on $ \mathcal{F}^\lambda(\C^{2n}) $  we will show that $ \varphi $ reduces to a constant. In view of the relation \eqref{relation:two-convolutions-operators}, which says that $ U^\ast \circ S_{U \varphi} \circ U = \widetilde{S}_\varphi$, we have a situation where both $ S_\varphi $ and $ S_{U\varphi} $ are bounded on $ \mathcal{F}^\lambda.$ In view of Theorem 1.1 both $ \varphi $ and $ U\varphi $ belong to $ \mathcal{A}_\lambda(\C^{2n}) $ and hence we  have a function $ f \in L^2(\R^{2n}) $ such that $ \pi_\lambda(f) = e^{-\frac{1}{2}H(\lambda)} M_1 $ and $ \pi_\lambda(U_\lambda f ) = e^{-\frac{1}{2}H(\lambda)} M_2 $ with $ M_1, M_2 \in B(L^2(\R^n)).$ Recall from \eqref{convolution-property-U-t-lambda} and \eqref{commuting-property-U-t-lambda} the following relation between $ f $ and $ U_\lambda f :$ 
$$U_\lambda f \ast_\lambda p_{1/2}^\lambda (x,u) = e^{-\frac{\lambda}{2}(\coth \lambda)(x^2+u^2)} \, f \ast_\lambda p_{1/2}^\lambda(-ix,-iu).$$
We would like to prove Theorem \ref{thm:uncertainty-twisted-fock-spaces} by making use of the following Hardy's theorem for the Weyl transform \cite[Theorem 2.9.5]{Book-Thangavelu-uncertainty} applied to the function $ g =  U_\lambda f \ast_\lambda p_{1/2}^\lambda.$ 

\begin{thm} \label{thm:Hardy-thm-Weyl-transform}
Let $ g \in L^2(\R^{2n}) $ be such that $ |g(x,u)| \leq C\, p_t^\lambda(x,u) $ and $ \pi_\lambda(g)^\ast \pi_\lambda(g) \leq C\, H(\lambda)^r\, e^{-2tH(\lambda)} $ for some constant $ C>0$ and $ r \in \mathbb{N}.$ Then $ g $ is a constant multiple of $ p_t^\lambda.$
\end{thm}

In order to make use of Theorem \ref{thm:Hardy-thm-Weyl-transform}, we need to check if our $ g $ satisfies both conditions appearing in Theorem \ref{thm:Hardy-thm-Weyl-transform} with $ t =1.$  The assumption $ \pi_\lambda(U_\lambda f ) = e^{-\frac{1}{2}H(\lambda)} M_2 $ gives us $ \pi_\lambda(g) = e^{-\frac{1}{2}H(\lambda)} M_2 e^{-\frac{1}{2}H(\lambda)} $ but unfortunately this does not give us an estimate of the form $ \pi_\lambda(g)^\ast \pi_\lambda(g) \leq C\, e^{-2H(\lambda)} $ unless $ M_2 $ commutes with $H(\lambda).$ However, getting a pointwise estimate on $ g $ is relatively easy.

\begin{prop} \label{prop:uncertainty-pointwise-bound}
With notations as above under the assumption  that $ M_1$  is a bounded linear operator on $ L^2(\R^n) ,$  we have  $ |g(x,u)| \leq C \, p_1^\lambda(x,u) .$ If we further assume that $ M_2 $ commutes with $ H(\lambda)$ we also have $ \pi_\lambda(g)^\ast \pi_\lambda(g) \leq C\, e^{-2H(\lambda)}. $
\end{prop}
\begin{proof} From the relation $ \pi_\lambda(f \ast_\lambda p_{1/2}^\lambda) = e^{-\frac{1}{2}H(\lambda)} M_1  e^{-\frac{1}{2}H(\lambda)} $ and the inversion formula for the Weyl transform we obtain
$$ f \ast_\lambda p_{1/2}^\lambda(-ix,-iu) = (2\pi)^{-n}|\lambda|^n \tr \left(\pi_\lambda(-ix,-iu) e^{-\frac{1}{2}H(\lambda)} M_1  e^{-\frac{1}{2}H(\lambda)}\right) .$$
We use the fact that the space  of trace class operators on $ L^2(\R^n) $ is a two sided ideal inside $ B(L^2(\R^n)),$ that is, for any trace class operator $ T $ and a bounded linear operator $ M $ we have 
$ |\tr( TM) | \leq \|M\|_{Op}\, \|T\|_1 = \|M\|_{Op} \, \tr (|T|).$ Observe that as $ \tr(AB) = \tr(BA) ,$
$$ \tr \left( e^{-\frac{1}{2}H(\lambda)} \pi_\lambda(-ix,-iu) e^{-\frac{1}{2}H(\lambda)} M_1 \right) = \tr \left( T_\lambda(x,u) M_1 \right) $$ 
and we can write  $ T_\lambda(x,u) =  S_\lambda(x,u)^\ast S_\lambda(x,u) $ by defining $ S_\lambda(x,u) = \pi_\lambda(\frac{1}{2}(-ix,-iu)) e^{-\frac{1}{2}H(\lambda)} .$  This  is a consequence of the fact that $ \pi_\lambda(x,u)  $ is a projective representation of $ \R^{2n} $ and $ \pi_\lambda(ix,iu) $ are selfadjoint. Thus $ T_\lambda(x,y) $ is positive and hence by the ideal property recalled above, we get (with $ C = \|M_1\|_{Op}$),
$$ | f \ast_\lambda p_{1/2}^\lambda(-ix,-iu)| \leq C \, \tr \left(\pi_\lambda(-ix,-iu) e^{-H(\lambda)}\right) = C\, p_1^\lambda(ix,iu).$$
Recalling the definition of $ g $ we get the estimate
$$ |g(x,u)| \leq  C \, e^{-\frac{\lambda}{2}(\coth \lambda)(x^2+u^2)} \, p_1^\lambda(-ix,-iu) = C  \, p_1^\lambda(x,u). $$
We  now show that under the extra assumption that $ M_2 $ is a function of $ H(\lambda) ,$  the  Weyl transform of the function $ g $ satisfies the required estimate. This is indeed very easy to check. As $ M_2 $ commutes with $ H(\lambda) $ we have $ \pi_\lambda(g) = M_2 e^{-H(\lambda)} $ and  hence for any $ \psi \in L^2(\R^n),$
$$ \langle \pi_\lambda(g)^\ast \pi_\lambda(g) \psi, \psi \rangle = \langle M_2^\ast M_2 e^{-H(\lambda)} \psi, e^{-H(\lambda)} \psi \rangle \leq \|M_2\|^2 \, \langle e^{-2H(\lambda)}\psi, \psi \rangle $$
which proves the required estimate. 
\end{proof}

As $ \pi_\lambda(U_\lambda f ) = e^{-\frac{1}{2}H(\lambda)} m_2(H(\lambda))  $ it follows that $ U_\lambda f $ is radial. Therefore, $ f $ is also radial since $ U_\lambda $ is essentially the Fourier transform. This means that $ M_1 = m_1(H(\lambda)) $ for some $ m _1 \in L^\infty(\R)$ and hence  $ \varphi(x,u) = G_\lambda(M_1)(x,u) $ is radial on $ \R^{2n}.$  Conversely, if the restriction of $ \varphi $ to $ \R^{2n} $ is radial, then $ M_1 $ is a function of $ H(\lambda) $ and we have the required estimates on $ g $ and $ \pi_\lambda(g) .$ Thus Theorem \ref{thm:uncertainty-twisted-fock-spaces} is proved in the particular case when $ \varphi $ is radial. \\ 

In order to deal with the general case we look at the action of the group $ K = O(2n,\R) \cap Sp(n,\R) $ on the twisted Fock space leading to a decomposition of the operators $ S_\varphi.$ Given  $ \sigma \in U(n) $ let us write $ \sigma = \alpha+i\beta $ where $ \alpha $ and $ \beta $ are real $ n \times n $ matrices. Then we can identity $ K $ with $U(n) $ and the action of $ \sigma $ on $ \R^{2n} $ is given by $ \sigma \cdot (x, u) = (\alpha x-\beta u,  \beta x+\alpha u).$ This action of $ U(n) $ can be naturally extended to $ \C^{2n} $ be defining $ \sigma \cdot (z, w) = (\alpha z-\beta w, \beta z+\alpha w).$ We record the following properties of this action.

\begin{lem} \label{lem:action-U(n)-bilinear-forms} 
The action of $ U(n) $ on $ \C^{2n} $ leaves the following bilinear forms invariant: 
(i) $B_1((z,w),(a,b)) = (z \cdot \bar{a}+w \cdot \bar{b}) $ and (ii) $B_2((z,w),(a,b)) = (w \cdot \bar{a}-z \cdot \bar{b}) .$
We also note that this action leaves the bilinear form $  Q(z,w) = \Im (z \cdot \bar{w}) $ on $ \C^{n} $ invariant.
\end{lem} 

For each $\sigma \in U(n) $ we consider the operator $ R_\sigma F(z,w) =  F(\sigma^{-1} \cdot (z,w)) $ on $ \mathcal{F}^\lambda(\C^{2n}).$ In view of Lemma \ref{lem:action-U(n)-bilinear-forms}, the weight function $ w_\lambda(z,w) $ is invariant under the action of $ U(n) $ and hence $ R_\sigma F \in \mathcal{F}^\lambda(\C^{2n}) $ which preserves the norm and so it defines a unitary operator. We also note that in view of Lemma \ref{lem:action-U(n)-bilinear-forms} and the definition of $ S_\varphi $ it follows that $ R_\sigma \circ S_\varphi \circ R_\sigma^\ast = S_{R_{\sigma}\varphi}.$ Thus we see that the radialisation $ S_\varphi^\#$ of the operator  $ S_\varphi $ is of the form $ S_{\varphi^\#} $ where
$$ S_\varphi^\#  = \int_{U(n)} R_\sigma \circ S_\varphi \circ R_\sigma^\ast  \, d\sigma  = S_{\varphi^\#}, \quad \textup{with } \, \, \varphi^\# = \int_{U(n)} R_\sigma \varphi \, d\sigma .$$
If $ \varphi = G_{\lambda}M,$ define $ f $ by the relation $ \pi_\lambda(f) = e^{-\frac{1}{2}H(\lambda)} M.$ Then it follows that $ \varphi^\# = G_{\lambda}M^\# $ where $ M^\# $ is defined by the relation $ \pi_\lambda(f^\#) = e^{-\frac{1}{2}H(\lambda)} M^\#.$ This is a consequence of the easily verified relation $ R_\sigma( f \ast_\lambda g) = R_\sigma f \ast_\lambda R_\sigma g.$ Thus we see that the functions $ \varphi $ which are radial (that is, invariant under the action of $ U(n)$) correspond to functions $ f $ on $ \R^{2n} $ (as given above) that are radial. For functions $ \varphi $ coming from $ \mathcal{A}^\lambda(\C^{2n}) $ this translates into the fact that the operator $ M $ appearing in the relation $ \pi_\lambda(f) = e^{-\frac{1}{2}H(\lambda)} M $ is a function of $ H(\lambda).$ Thus, in Proposition \ref{prop:uncertainty-pointwise-bound} we have taken care of the case when $ \varphi $ is radial.\\

To deal with the general case we express $ \varphi $ as a superposition of $ \varphi_\delta $ indexed by class one representations of $ K.$ For each $ \delta $ we will show that both $ S_{\varphi_\delta} $ and $ \widetilde{S}_{\varphi_\delta} $ are bounded on $ \mathcal{F}^\lambda(\C^{2n}) $ and consequently $ \varphi_\delta = 0$ for each $ \delta $ except for the trivial representation which corresponds to $ \varphi^\# $  which is already treated in the Proposition \ref{prop:uncertainty-pointwise-bound}. \\

With notations as in \Cref{subsec-prelim-operator-sph-harmonics}, for each class one representation $ \delta $ of $ K = U(n) $ let $ \chi_\delta $ stand for its character. We identify $ L^2(S^{2n-1}) $ with $ L^2(K/M) $ where $ M  \subset K $ is the isotropic subgroup fixing the coordinate vector $ e_1 \in \C^n.$ Note that we can identity $ M $ with $ U(n-1) $ and functions on $ S^{2n-1} $ are in one to one correspondence with right $ M $ invariant functions on $K.$  In view of Peter-Weyl theorem for the compact symmetric space $ K/M $ we obtain
$$ \varphi(z,w) = \sum_{\delta \in \widehat{K_0}} \varphi_\delta(z,w), $$ 
where 
$$ \varphi_\delta(z,w) = \int_{U(n)} R_\sigma \varphi (z,w) \, \chi_\delta(\sigma^{-1}) \, d\sigma .$$
In view of the relation $ R_\sigma \circ S_\varphi \circ R_\sigma^\ast = S_{R_{\sigma}\varphi} $ it follows that 
$$ S_{\varphi_\delta} = \int_{U(n)} R_\sigma \circ S_\varphi \circ R_\sigma^\ast \,\,\, \chi_\delta(\sigma^{-1}) \, d\sigma .$$
Consequently, boundedness of $ S_\varphi $ implies boundedness of $ S_{\varphi_\delta} $ for every $ \delta \in \widehat{K_0}.$ Moreover,  it follows directly from the definition of $U^*$ that $(U^* \varphi)_\delta = U^* \varphi_\delta,$ and therefore the boundedness of $ S_{(U^* \varphi)_\delta}$ follows from that of $S_{U^* \varphi}.$
\\

We show that the simultaneous boundedness of $ S_{\varphi_\delta} $ and  $ S_{(U^* \varphi)_\delta}$ forces $\varphi_\delta $ to be zero whenever $ \delta $ is non-trivial. This will complete the proof of Theorem \ref{thm:uncertainty-twisted-fock-spaces}. \\

Recall that when $S_\varphi$ is bounded, we have  a function $ f \in L^2(\R^{2n}) $ satisfying $ \pi_\lambda(f) = e^{-\frac{1}{2}H(\lambda)} M_1 $ with $ M_1 \in B(L^2(\R^n)),$ such that  
$$\varphi (z,w) = e^{\frac{\lambda}{4}(\coth 2 \lambda)(z^2+w^2)} \, f \ast_\lambda p_{1/2}^\lambda (z,w) $$ 
and, as earlier, let us define $g =  U_\lambda f \ast_\lambda p_{1/2}^\lambda .$ It then follows that 
$$\varphi_\delta (z,w) = e^{\frac{\lambda}{4}(\coth 2 \lambda)(z^2+w^2)} \, f_\delta \ast_\lambda p_{1/2}^\lambda (z,w) .$$ 
Moreover, as the Euclidean Fourier transform commutes with rotations, it follows that $(U_\lambda f)_\delta = U_\lambda f_\delta.$ As a consequence,   $g_\delta =  U_\lambda f_\delta \ast_\lambda p_{1/2}^\lambda$ and hence the estimate on $ g $ proved in Proposition \ref{prop:uncertainty-pointwise-bound} yields the same estimate on $ g_\delta.$  In order to calculate $ \pi_\lambda(g_\delta)$ we proceed as follows. \\

For any $ h \in L^2(\R^{2n}), \pi_\lambda(R_\sigma h) = \mu_\lambda(\sigma) \pi_\lambda(h) \mu_\lambda(\sigma)^* $ and hence making use of the fact that $ \mu_\lambda(\sigma) $ commutes with $ e^{-tH(\lambda)} $ for any $ \sigma \in U(n) ,$ an easy calculation shows  that 
$$ \pi_\lambda(g_\delta) = e^{-\frac{1}{2}H(\lambda)} M_2^\delta \,e^{-\frac{1}{2}H(\lambda)}$$
where we have defined $ M_2^\delta $ by the equation
$$ M_2^\delta = \int_{U(n)} \mu_\lambda(\sigma) M_2 \,   \mu_\lambda(\sigma)^*\, \chi_\delta(\sigma^{-1}) \, d\sigma .$$ 
We now make use of the results stated in Subsection \ref{subsec-prelim-operator-sph-harmonics}. Expanding $ M_2^\delta $ in terms of  $ S_{j,k}^\delta $ we have
$$ M_2^\delta = \sum_{j=1}^{d(\delta)} W_\lambda(P_j^\delta) \, m_j (H(\lambda))$$
where $m_j \in L^\infty (\R)$ for each $1 \leq j \leq d(\delta)$. We thus have
$$ \pi_\lambda(g_\delta) = \sum_{j=1}^{d(\delta)} e^{-\frac{1}{2}H(\lambda)} \,W_\lambda(P_j^\delta) \, e^{-\frac{1}{2}H(\lambda)}\, m_j (H(\lambda)). $$ 
By defining $ V_j^\delta(\lambda) = [e^{-\frac{1}{2}H(\lambda)}, W_\lambda(P_j^\delta)] \,e^{\frac{1}{2}H(\lambda)}$ we can write the above as 
$$ \pi_\lambda(g_\delta) = \sum_{j=1}^{d(\delta)} \left( W_\lambda(P_j^\delta) + V_j^\delta(\lambda)\right) \, m_j (H(\lambda))\,e^{-H(\lambda)}. $$ 
In view of a result of Douglas \cite[Theorem 1]{Douglas-factorisation-PAMS}, it is enough to show that the linear operator $ \left( W_\lambda(P_j^\delta) + V_j^\delta(\lambda)\right) H(\lambda)^{-r} $ is bounded on $ L^2(\R^n) $  for some $ r \in \mathbb{N}.$ \\ 

Let $ \delta= \delta(p,q)$  and consider $ P^\delta(z) = z_1^p \, \bar{z}_2^q.$ Since the space $\mathcal{H}_{p,q}$ is irreducible under the action of $U(n)$, it follows that for each $1 \leq j \leq d(\delta)$, there exist $\sigma_{j,k} \in U(n)$ and constants $ c_{j,k} $ such that 
$$ P_j^\delta = \sum_{k=1}^{m_j} c_{j,k}\, R_{\sigma_{j,k}} P^\delta, $$ 
and therefore 
$$ W_\lambda(P_j^\delta) = \sum_{k=1}^{m_j} c_{j,k}\, \mu_{\lambda}(\sigma_{j,k}) W_\lambda ( P^\delta) \mu_{\lambda}(\sigma_{j,k})^\ast. $$ 
In view of the above, as $ \mu_\lambda(\sigma)$ commutes with $ e^{-\frac{1}{2}H(\lambda)}, $ it is enough to show the boundedness of $ \left( W_\lambda(P^\delta) + V^\delta(\lambda)\right) H(\lambda)^{-r} $ where $V^\delta(\lambda)$ is defined the same way as $V_j^\delta(\lambda).$\\

The boundedness of $ W_\lambda(P^\delta)H(\lambda)^{-r} $ with $ r = (p+q)/2$ is easy to see. Indeed, for $ \lambda >0,$ we have $  W_\lambda(P^\delta) = A_2^\ast(\lambda)^q A_1(\lambda)^p$ and hence 
$$W_\lambda(P^\delta)H(\lambda)^{-(p+q)/2} \Phi_\alpha^\lambda =  (2|\alpha|+n)^{-(p+q)/2} A_2^\ast(\lambda)^q A_1(\lambda)^p \Phi_\alpha^\lambda .$$
Recalling the action of the creation and annihilation operators on Hermite functions, viz.,
$$A_2^\ast(\lambda) \Phi_\alpha^\lambda = \left( \left( 2 \alpha_2 + 2 \right) |\lambda| \right)^{1/2} \Phi_{\alpha + e_2}^\lambda, \quad A_1(\lambda) \Phi_\alpha^\lambda = \left( \left( 2 \alpha_1 \right) |\lambda| \right)^{1/2} \Phi_{\alpha - e_1}^\lambda$$
we see that $ A_2^\ast(\lambda)^q A_1(\lambda)^p \Phi_\alpha^\lambda = c_{\alpha}^\delta(\lambda)\, \Phi_{\alpha - pe_1 + qe_2}^\lambda $ with $ | c_{\alpha}^\delta(\lambda)| \leq c \, \left( \left( (2|\alpha|+n) \right) |\lambda| \right)^{(p+q)/2}.$ This 
proves the boundedness of $ W_\lambda(P^\delta) H(\lambda)^{-(p+q)/2}.$ A similar argument also proves the boundedness of $V^\delta(\lambda) H(\lambda)^{-(p+q)/2}. $ Hence we have proved the estimate 
$$ \pi_\lambda(g_\delta)^\ast \pi_\lambda(g_\delta) \leq C\, H(\lambda)^{p+q}\, e^{-2H(\lambda)}.$$
This completes the proof of the uncertainty principle as stated in \Cref{thm:uncertainty-twisted-fock-spaces}. \\

%%%%%%%%%%%%%%%%%%%%%%%%%%%%%%%%%%%%%%
%%%%%%%%%%%%%%%%%%%%%%%%%%%%%%%%%%%%%%
%%%%%%%%%%%%%%%%%%%%%%%%%%%%%%%%%%%%%%
%%%%%%%%%%%%%%%%%%%%%%%%%%%%%%%%%%%%%%
%%%%%%%%%%%%%%%%%%%%%%%%%%%%%%%%%%%%%%
%%%%%%%%%%%%%%%%%%%%%%%%%%%%%%%%%%%%%%
%%%%%%%%%%%%%%%%%%%%%%%%%%%%%%%%%%%%%%
%%%%%%%%%%%%%%%%%%%%%%%%%%%%%%%%%%%%%%
%%%%%%%%%%%%%%%%%%%%%%%%%%%%%%%%%%%%%%
%%%%%%%%%%%%%%%%%%%%%%%%%%%%%%%%%%%%%%

\section*{Acknowledgements}
This work was initiated when the second author (ST) visited IISER, Bhopal in March 2023 as Adjunct Faculty and completed in June 2023 when the first author (RG) visited IISc, Bangalore. RG and ST wish to thank their host institutes for the hospitality. ST was supported in part by DST (J. C. Bose Fellowship) and INSA.\\

%%%%%%%%%%%%%%%%%%%%%%%%%%%%%%%%%%%%%%
%%%%%%%%%%%%%%%%%%%%%%%%%%%%%%%%%%%%%%
%%%%%%%%%%%%%%%%%%%%%%%%%%%%%%%%%%%%%%
%%%%%%%%%%%%%%%%%%%%%%%%%%%%%%%%%%%%%%
%%%%%%%%%%%%%%%%%%%%%%%%%%%%%%%%%%%%%%
%%%%%%%%%%%%%%%%%%%%%%%%%%%%%%%%%%%%%%
%%%%%%%%%%%%%%%%%%%%%%%%%%%%%%%%%%%%%%
%%%%%%%%%%%%%%%%%%%%%%%%%%%%%%%%%%%%%%
%%%%%%%%%%%%%%%%%%%%%%%%%%%%%%%%%%%%%%
%%%%%%%%%%%%%%%%%%%%%%%%%%%%%%%%%%%%%%

\providecommand{\bysame}{\leavevmode\hbox to3em{\hrulefill}\thinspace}
\providecommand{\MR}{\relax\ifhmode\unskip\space\fi MR }
\providecommand{\MRhref}[2]{%
  \href{http://www.ams.org/mathscinet-getitem?mr=#1}{#2}
}
\providecommand{\href}[2]{#2}


\begin{thebibliography}{10}

\bibitem{Basak-Garg-Thangavelu-Weyl}
Riju Basak, Rahul Garg, and Sundaram Thangavelu, \emph{Homogeneous {F}ourier
  and {W}eyl multipliers on {S}obolev spaces related to the {H}eisenberg
  group}, J. Funct. Anal. \textbf{281} (2021), no.~8, Paper No. 109154, 45.
  \MR{4280273}

\bibitem{CHLS}
Guangfu Cao, Li~He, Ji~Li, and Minxing Shen, \emph{Boundedness criterion for
  integral operators on the fractional {F}ock-{S}obolev spaces}, Math. Z.
  \textbf{301} (2022), no.~4, 3671--3693. \MR{4449725}

\bibitem{CLSWY}
Guangfu Cao, Ji~Li, Minxing Shen, Brett~D. Wick, and Lixin Yan, \emph{A
  boundedness criterion for singular integral operators of convolution type on
  the {F}ock space}, Adv. Math. \textbf{363} (2020), 107001, 33. \MR{4053517}

\bibitem{Douglas-factorisation-PAMS}
Ronald~G. Douglas, \emph{On majorization, factorization, and range inclusion of
  operators on {H}ilbert space}, Proc. Amer. Math. Soc. \textbf{17} (1966),
  413--415. \MR{203464}

\bibitem{Book-Folland-phase-space}
Gerald~B. Folland, \emph{Harmonic analysis in phase space}, Annals of
  Mathematics Studies, vol. 122, Princeton University Press, Princeton, NJ,
  1989. \MR{983366}

\bibitem{Geller-spherical-harmonics-weyl-transform}
Daryl Geller, \emph{Spherical harmonics, the {W}eyl transform and the {F}ourier
  transform on the {H}eisenberg group}, Canad. J. Math. \textbf{36} (1984),
  no.~4, 615--684. \MR{756538}

\bibitem{KTX}
Bernhard Kr\"{o}tz, Sundaram Thangavelu, and Yuan Xu, \emph{The heat kernel
  transform for the {H}eisenberg group}, J. Funct. Anal. \textbf{225} (2005),
  no.~2, 301--336. \MR{2152501}

\bibitem{MauceriWeylTransformJFA80}
Giancarlo Mauceri, \emph{The {W}eyl transform and bounded operators on
  {$L^{p}({\bf R}^{n})$}}, J. Funct. Anal. \textbf{39} (1980), no.~3, 408--429.
  \MR{600625}

\bibitem{Thangavelu-Hermite-Laguerre-Expansions-Book}
Sundaram Thangavelu, \emph{Lectures on {H}ermite and {L}aguerre expansions},
  Mathematical Notes, vol.~42, Princeton University Press, Princeton, NJ, 1993,
  With a preface by Robert S. Strichartz. \MR{1215939}

\bibitem{Book-Thangavelu-uncertainty}
\bysame, \emph{An introduction to the uncertainty principle}, Progress in
  Mathematics, vol. 217, Birkh\"{a}user Boston, Inc., Boston, MA, 2004, Hardy's
  theorem on Lie groups, With a foreword by Gerald B. Folland. \MR{2008480}

\bibitem{Thangavelu-Poisson-trans-Heisenberg}
\bysame, \emph{Poisson transform for the {H}eisenberg group and eigenfunctions
  of the sublaplacian}, Math. Ann. \textbf{335} (2006), no.~4, 879--899.
  \MR{2232020}

\bibitem{Thangavelu-arxiv-Fock-Sobolev-2023}
\bysame, \emph{Fourier multipliers and pseudo-differential operators on
  {F}ock-{S}obolev spaces}, https://arxiv.org/abs/2304.01087 (2023).

\end{thebibliography}
\end{document}